\documentclass[reqno, 11pt]{amsart}
\usepackage{etex}
\usepackage[utf8]{inputenc}
\usepackage[english]{babel}
\usepackage{amssymb,amsmath,amsthm}
\usepackage{graphicx,geometry,caption}
\usepackage{float}
\usepackage{enumerate}
\usepackage{tikz}
\usetikzlibrary{patterns}
\usepackage{pgfplots}
\usepackage[all]{xy}
\usepackage{tikz-cd}
\usepackage{comment}
\usepackage{cite}
\usepackage{setspace}
\usepackage{amssymb,amstext}
\usepackage{xcolor}
\usepackage{caption}
\pgfplotsset{compat=1.10}

\numberwithin{equation}{section}
\newcommand{\extp}{\@ifnextchar^\@extp{\@extp^{\,}}}
\def\extp^#1{\mathop{\bigwedge\nolimits^{\!#1}}}
\makeatother

\theoremstyle{plain}
\newtheorem{theorem}{Theorem}[section]
\newtheorem{proposition}{Proposition}[section]

\newtheorem{lemma}{Lemma}[section]

\theoremstyle{definition}

\newtheorem{remark}{Remark}[section]

\begin{document}

\title{The Rohlin invariant and $\mathbb{Z}/2$-valued invariants of homology spheres}

% author information
\author{Ricard Riba}
\address{Universitat Autònoma de Barcelona, Departament de Matemàtiques, Bellaterra, Spain}
\email{riba@mat.uab.cat}
\thanks{This work was partially supported by MEC grant MTM2016-80439-P}



\date{\today}

%The Rohlin invariant and the algebraic construction of invariants of homology spheres

\begin{abstract}
In this paper we prove that the Rohlin invariant is the unique invariant inducing a homomorphism on the Torelli group. Using this result we generalize the construction of invariants of homology $3$-spheres from families of trivial 2-cocycles on the Torelli group given by Pitsch to include invariants with values on an abelian group with $2$-torsion.
\end{abstract}

\maketitle

\section{Introduction}

Let $\Sigma_g$ be an oriented surface of genus $g$ standardly embedded in the oriented $3$-sphere $\mathbf{S}^3$. Denote by $\Sigma_{g,1}$ the complement of the interior of a small disk embedded in $\Sigma_g$. 
The surface $\Sigma_g$ separates $\mathbf{S}^3$ into two genus $g$ handlebodies $\mathbf{S}^3=\mathcal{H}_g\cup -\mathcal{H}_g$ with opposite induced orientation.
Denote by $\mathcal{M}_{g,1}$ the mapping class group of $\Sigma_{g,1},$ ie. the group of isotopy classes of orientation-preserving diffeomorphisms of $\Sigma_g$ which are the identity on the disk modulo isotopies which again fix that small disk point-wise.
The embedding of $\Sigma_g$ in $\mathbf{S}^3$ determines three natural subgroups of $\mathcal{M}_{g,1}$: the subgroup $\mathcal{B}_{g,1}$ of mapping classes that extends to the inner handlebody
$\mathcal{H}_g$, the subgroup $\mathcal{A}_{g,1}$ of mapping classes that extends to the outer handlebody $-\mathcal{H}_g,$ and their intersection $\mathcal{AB}_{g,1}$.

By the theory of Heegaard splittings we know that any element in $\mathcal{S}^3,$ the set of diffeomorphism classes of compact, closed and oriented smooth homology $3$-spheres, can be obtained by cutting $\mathbf{S}^3$ along $\Sigma_g$ for some $g$ and glueing back the two handlebodies by some element of the Torelli group $\mathcal{T}_{g,1},$ which is the group formed by the elements of the mapping class group $\mathcal{M}_{g,1}$ that act trivially on the first homology group of the surface $\Sigma_{g}.$

In \cite{pitsch}, Pitsch proved that the lack of injectivity of this construction is controlled by the subgroups $\mathcal{TB}_{g,1}=\mathcal{T}_{g,1}\cap \mathcal{B}_{g,1},$ $\mathcal{TA}_{g,1}=\mathcal{T}_{g,1}\cap \mathcal{A}_{g,1},$ and the conjugation by elements of the group $\mathcal{AB}_{g,1}$. More precisely the injection $\Sigma_{g,1} \hookrightarrow \Sigma_{g+1,1}$ induces a natural injective stabilization map $\mathcal{T}_{g,1}\hookrightarrow \mathcal{T}_{g+1,1}$, which is compatible with the definitions of the above subgroups and one gets a well-defined bijective map:
\begin{equation}
\label{bij_S3-intro}
\lim_{g\to \infty}(\mathcal{TA}_{g,1}\backslash\mathcal{T}_{g,1}/\mathcal{TB}_{g,1})_{\mathcal{AB}_{g,1}} \xrightarrow{\sim}  \mathcal{S}^3.
\end{equation}
As a consequence of this bijection, every invariant $F$ of homology $3$-spheres with values in an arbitrary abelian group $A$ can be viewed as a family of functions $\lbrace F_g\rbrace_g$ from the Torelli groups $\mathcal{T}_{g,1}$ to the abelian group $A$ satisfying the following properties:

\begin{enumerate}[i)]
\item $F_{g+1}(x)=F_g(x) \quad \text{for every }x\in \mathcal{T}_{g,1},$
\item $F_g(\xi_a x\xi_b)=F_g(x) \quad \text{for every } x\in \mathcal{T}_{g,1},\;\xi_a\in \mathcal{TA}_{g,1},\;\xi_b\in \mathcal{TB}_{g,1},$
\item $F_g(\phi x \phi^{-1})=F_g(x)  \quad \text{for every }  x\in \mathcal{T}_{g,1}, \; \phi\in \mathcal{AB}_{g,1}.$
\end{enumerate}

Using this reformulation of invariants of homology $3$-spheres in terms of family of functions on the Torelli group, in the aforementioned paper Pitsch gave a tool to construct invariants of integral homology $3$-spheres, with values in any abelian group without $2$-torsion from a family of trivial $2$-cocycles on the Torelli group $\mathcal{T}_{g,1}$ with an $\mathcal{AB}_{g,1}$-invariant trivialization and satisfying the following conditions:

\begin{enumerate}[(1)]
\item The $2$-cocycles $\{C_g\}_g$ are compatible with the stabilization map,
\item The $2$-cocycles $\{C_g\}_g$ are invariant under conjugation by elements in $\mathcal{AB}_{g,1},$
\item If $\phi\in \mathcal{TA}_{g,1}$ or $\psi \in \mathcal{TB}_{g,1}$ then $C_g(\phi, \psi)=0.$
\end{enumerate}

In this paper we generalize this construction of invariants of homology $3$-spheres from trivial $2$-cocycles to include values in an abelian group with $2$-torsion.

The main difficulty to generalize this construction is the existence of non-zero homomorphisms on the Torelli group that satisfy the aforementioned properties i)-iii) and therefore reassemble to invariants of homology $3$-spheres.
Nevertheless we show that the homomorphisms induced by the Rohlin invariant are the unique ones with such properties.

\begin{theorem}
\label{teo-roh-intro}
Let $A$ be an abelian group and $A_2$ the subgroup of $2$-torsion elements. Provided $g\geq 4,$ the BCJ-homomorphism $\sigma: \mathcal{T}_{g,1}\rightarrow \mathfrak{B}_3$ composed with the projection $\pi_g:\mathfrak{B}_3\rightarrow\mathfrak{B}_0\cong \mathbb{Z}/2$ and the injection $\varepsilon^x:\mathfrak{B}_0\rightarrow A_2$ defined by sending $\overline{1}$ to $x,$ induces an isomorphism
\begin{align*}
\Lambda: A_2 & \longrightarrow Hom(\mathcal{T}_{g,1},A)^{\mathcal{AB}_{g,1}} \\
x & \longmapsto \mu^x_g= \varepsilon^{x}\circ \pi_g \circ\sigma.
\end{align*}
Moreover the family of homomorphisms $\{\mu_g^x\}_g$ reassemble into the Rohlin invariant.
\end{theorem}

This result allows us to generalize the construction of invariants of homology $3$-spheres given by Pitsch in \cite{pitsch} to include values on an abelian group with $2$-torsion.

\begin{theorem}
\label{teo_gen_tool-intro}
Let $A$ be an abelian group and $A_2$ the subgroup of $2$-torsion elements.
For each $x\in A_2$ and $g\geq 4,$ a family of cocycles $\{C_g\}_{g\geq 3}$ on the Torelli group $\mathcal{T}_{g,1}$ satisfying conditions (1)-(3) provides a compatible familiy of trivializations $F_g+\mu_g^x: \mathcal{T}_{g,1}\rightarrow A$ that reassemble into an invariant of integral homology $3$-spheres
$$\lim_{g\to \infty}F_g+\mu_g^x:\mathcal{S}^3\rightarrow A$$
if and only if the following two conditions hold:
\begin{enumerate}[(i)]
\item The associated cohomology classes $[C_g]\in H^2(\mathcal{T}_{g,1};A)$ are trivial,
\item The associated torsors $\rho(C_g)\in H^1(\mathcal{AB}_{g,1},Hom(\mathcal{T}_{g,1},A))$ are trivial.
\end{enumerate}
\end{theorem}

\textbf{Outlines of this work}

In Section 2 we review some definitions about the mapping class group, the symplectic representation, the Boolean algebra, the Birman-Craggs-Johnson-homomoprhism and handlebodies. We also compute the $GL_g(\mathbb{Z})$-coinvariants of the Boolean algebras of degree $2,$ $3$ and give a basic 3-dimensional topology lemma about handlebodies.
In Section 3 we recall the definitions of contractible bounding pair twists and the Luft group. Moreover, we exhibit some results about the handlebody subgroup $\mathcal{B}_{g,1}$ and the Luft-Torelli group $\mathcal{LTB}_{g,1}$.
In Section 4 we prove that $H_1(\mathcal{T}_{g,1};\mathbb{Z})_{\mathcal{AB}_{g,1}}$ is isomorphic to $\mathbb{Z}/2,$ and we also show that $H_1(\mathcal{TB}_{g,1};\mathbb{Z})_{\mathcal{AB}_{g,1}}$ and  $H_1(\mathcal{TA}_{g,1};\mathbb{Z})_{\mathcal{AB}_{g,1}}$ are zero. At the end of this section, using the aforementioned computations, we prove Theorem \ref{teo-roh-intro}.
Finally, in Section 5 we give the proof of Theorem
\ref{teo_gen_tool-intro}.

\section{Preliminaries}

\subsection{The Symplectic representation}
For a given integer $g\geq 1$ consider the basis $\{ a_1,\ldots , a_g, b_1,\ldots, b_g \}$ of $H_1(\Sigma_{g,1};\mathbb{Z})$ given by the homology class of the respective curves $\{ \alpha_1,\ldots , \alpha_g, \beta_1,\ldots, \beta_g \}$ depicted in Fig.~\ref{fig:homology_basis}. Transverse intersection of oriented paths on $\Sigma_{g,1}$ induces a symplectic form $\omega$ on $H_1(\Sigma_{g,1};\mathbb{Z})$, with $\omega(b_i,a_i)=-\omega(a_i,b_i)=1$ and zero otherwise. Moreover, both sets of displayed homology classes $\{ a_i \ | \ 1 \leq i \leq g\}$ and $\{ b_i \ | \ 1 \leq i \leq g \}$ form a symplectic basis, and in particular generate supplementary transverse Lagrangians $A$ and $B$. As a symplectic space we write $H_1(\Sigma_{g,1};\mathbb{Z}) = A \oplus B$.

\begin{figure}[H]
\begin{center}
\begin{tikzpicture}[scale=.7]
\draw[very thick] (-4.5,-2) -- (5,-2);
\draw[very thick] (-4.5,2) -- (5,2);
\draw[very thick] (-4.5,2) arc [radius=2, start angle=90, end angle=270];

\draw[very thick] (-4.5,0) circle [radius=.4];
\draw[very thick] (-2,0) circle [radius=.4];
\draw[very thick] (2,0) circle [radius=.4];

\draw[thick, dotted] (-0.5,0) -- (0.5,0);

\draw[thick] (1.2,0) to [out=90,in=180] (2,0.8);
\draw[thick] (2.8,0) to [out=90,in=0] (2,0.8);
\draw[thick] (1.2,0) to [out=-90,in=180] (2,-0.8) to [out=0,in=-90] (2.8,0);

\draw[thick] (-5.3,0) to [out=90,in=180] (-4.5,0.8);
\draw[thick] (-3.7,0) to [out=90,in=0] (-4.5,0.8);
\draw[thick] (-5.3,0) to [out=-90,in=180] (-4.5,-0.8) to [out=0,in=-90] (-3.7,0);

\draw[thick] (-2.8,0) to [out=90,in=180] (-2,0.8);
\draw[thick] (-1.2,0) to [out=90,in=0] (-2,0.8);
\draw[thick] (-2.8,0) to [out=-90,in=180] (-2,-0.8) to [out=0,in=-90] (-1.2,0);

\draw[thick,dashed] (-2,0.4) to [out=70,in=-70] (-2,2);
\draw[thick] (-2,0.4) to [out=110,in=-110] (-2,2);

\draw[thick,dashed] (-4.5,0.4) to [out=70,in=-70] (-4.5,2);
\draw[thick] (-4.5,0.4) to [out=110,in=-110] (-4.5,2);

\draw[thick,dashed] (2,0.4) to [out=70,in=-70] (2,2);
\draw[thick] (2,0.4) to [out=110,in=-110] (2,2);

\node at (-5.1,1.2) {$\beta_1$};
\node [below] at (-4.5,-0.8) {$\alpha_1$};
\node at (-2.6,1.2) {$\beta_2$};
\node at (1.4,1.2) {$\beta_g$};
\node [below] at (-2,-0.8) {$\alpha_2$};
\node [below] at (2,-0.8) {$\alpha_g$};

\draw[thick,pattern=north west lines] (5,-2) to [out=130,in=-130] (5,2) to [out=-50,in=50] (5,-2);

\draw [<-](1.2,0.01) -- (1.2,0);
\draw [<-](-2.8,0.01) -- (-2.8,0);
\draw [<-](-5.3,0.01) -- (-5.3,0);

\draw [->](-4.65,1.21) -- (-4.65,1.2);
\draw [->](-2.15,1.21) -- (-2.15,1.2);
\draw [->](1.85,1.21) -- (1.85,1.2);

\end{tikzpicture}
\end{center}
\caption{Symplectic basis of $H_1(\Sigma_{g,1};\mathbb{Z})$.}
\label{fig:homology_basis}
\end{figure}
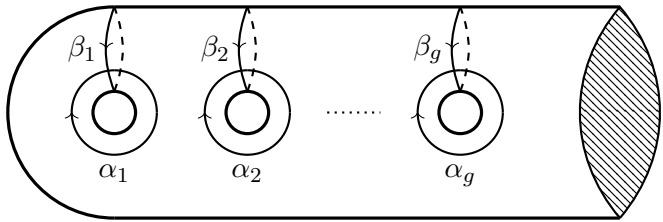

The symplectic form $\omega$ is preserved by the natural action of the mapping class group on the first homology of $\Sigma_{g,1}$ and gives rise to the symplectic representation
\[
\Psi: \mathcal{M}_{g,1}\longrightarrow Sp_{2g}(\mathbb{Z}),
\]
which is known to be a surjective map (cf. \cite{burk}).
In particular, by \cite[Lemma 3]{pitsch}, the action of $\mathcal{AB}_{g,1}$ on $H_1(\Sigma_{g,1};\mathbb{Z})$ factors through $GL_g(\mathbb{Z}),$
where an element $G\in GL_g(\mathbb{Z})$ acts on $H_1(\Sigma_{g,1};\mathbb{Z})$ via the action of the matrix $\left(\begin{smallmatrix}
G & 0 \\
0 &{}^tG^{-1}
\end{smallmatrix}\right)\in Sp_{2g}(\mathbb{Z}).$

To understand better $\Psi$ we show what it does to Dehn twists: given $k\in \mathbb{Z}$ and two oriented simple closed
curves $x$ and $y$ in $\Sigma_{g,1}$ with respective homology classes $[x]$ and $[y],$ the action of the Dehn twist $T^k_x$ on $H_1(\Sigma_{g,1};\mathbb{Z})$ is given by:
$$\Psi(T^k_x)([y])=[y]+k\;\omega([y],[x])[x].$$

To avoid too heavy notation in all subsequent sections we will sometimes abbreviate $H:=H_1(\Sigma_{g,1};\mathbb{Z})$ and $H_2:=H_1(\Sigma_{g,1};\mathbb{Z}/2).$

\subsection{The Boolean algebra and the BCJ-homomorphism}

The Boolean polynomial algebra $\mathfrak{B}=\mathfrak{B}_{g}$ is a $\mathbb{Z}/2$-algebra (with unit $1$) with a generator $\overline{x}$ for each $x\in H_1(\Sigma_{g,1};\mathbb{Z}/2)$ and subject to the relations:
\begin{enumerate}[(a)]
\item $\overline{x+y}=\overline{x}+\overline{y}+ \omega(x, y) \text{ mod } 2,$
\item $\overline{x}^2=\overline{x}.$
\end{enumerate}
The relation (b) implies that $p^2=p$ for any $p\in \mathfrak{B}$ and also that if $\{e_i \mid i\in \{ 1,\ldots 2g\}\}$ is a basis for $H_1(\Sigma_{g,1};\mathbb{Z}/2)$ then the set of all monomials
$\overline{e}_{i_1}\overline{e}_{i_2}\ldots \overline{e}_{i_r}$ with $0\leq r\leq 2g$ and $1\leq i_1< i_2< \ldots < i_r\leq 2g$ is a $\mathbb{Z}/2$-basis for $\mathfrak{B}.$ Denote by  $\mathfrak{B}_k=\mathfrak{B}^k_{g}$ the subspace generated by all monomials of ``degree'' $\leq k.$

In \cite{jo1}, Johnson constructed a surjective homomorphism $\sigma: \mathcal{T}_{g,1}\rightarrow \mathfrak{B}_3,$ called the Birman-Craggs-Johnson homomorphism (abbreviated BCJ-homomorphism), which may be described as follows.
Consider the $\mathbb{Z}/2$-basis of $\mathfrak{B}_3$ given by
$$
\left\{\overline{1},\overline{a}_i,\overline{b}_i,\overline{a}_i\overline{a}_j,\overline{b}_i\overline{b}_j, \overline{a}_i\overline{b}_j, \overline{a}_i\overline{b}_i,\overline{a}_i\overline{a}_j\overline{a}_k,\overline{b}_i\overline{b}_j\overline{b}_k, \overline{a}_i\overline{b}_j\overline{b}_k,\overline{a}_i\overline{a}_j\overline{b}_k,\overline{a}_i\overline{b}_i\overline{b}_j, \overline{a}_i\overline{b}_i\overline{a}_j \right\},$$
where $i,j,k\in \{1,\ldots, g\}$ are pairwise distinct.
Consider the curves depicted in the following figure:

\begin{figure}[H]
\begin{center}
\begin{tikzpicture}[scale=.7]
\draw[very thick] (-4.5,-2) -- (7,-2);
\draw[very thick] (-4.5,2) -- (7,2);
\draw[very thick] (-4.5,2) arc [radius=2, start angle=90, end angle=270];
\draw[very thick] (-4.5,0) circle [radius=.4];
\draw[thick, dotted] (-3.5,0) -- (-3,0);
\draw[very thick] (-2,0) circle [radius=.4];
\draw[very thick] (2,0) circle [radius=.4];
\draw[thick, dotted] (3.5,0) -- (3,0);
\draw[very thick] (4.5,0) circle [radius=.4];

\draw[thick, dashed] (0,2) to [out=-70,in=70] (0,-2);
\draw[thick ] (0,2) to [out=-110,in=110] (0,-2);

\draw[thick,dashed] (2,0.4) to [out=70,in=-70] (2,2);
\draw[thick] (2,0.4) to [out=110,in=-110] (2,2);

\draw[thick,dashed] (2,-0.4) to [out=-70,in=70] (2,-2);
\draw[thick] (2,-0.4) to [out=-110,in=110] (2,-2);

\node[right] at (2.1,1.5) {$\beta$};
\node[right] at (2.1,-1.5) {$\beta'$};
\node at (0,0) {$\gamma$};

\node at (-4.5,0) {\tiny{1}};
\node at (-2,0) {\tiny{k}};
\node at (2,0) {\tiny{k+1}};
\node at (4.5,0) {\tiny{g}};

\draw[thick,pattern=north west lines] (7,-2) to [out=130,in=-130] (7,2) to [out=-50,in=50] (7,-2);

\end{tikzpicture}
\end{center}
\caption{A bounding simple closed curve and a bounding pair in $\Sigma_{g,1}$.}
\end{figure}
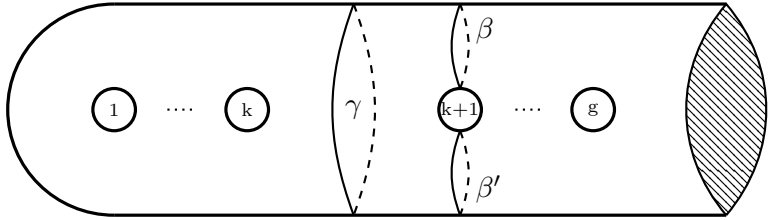

The BCJ-homomorphism is given on a BP-map $T_\beta T_{\beta'}^{-1}$ by (cf. \cite[Lemma 12b]{jo1})
\[\sigma(T_\beta T_{\beta'}^{-1})= \sum^k_{i=1}\overline{c}_i\overline{d}_i(\overline{e}+\overline{1}),\]
where $e$ is the homology class of $\beta,$ and $\{c_i,d_i\}_i$ is a symplectic basis of a subsurface $\Sigma_{k,1}$ of $\Sigma_{g,1}$ with boundary component $\gamma,$ such that $\gamma \cup \beta \cup \beta'$ is the boundary of a subsurface with genus zero in $\Sigma_{g,1}.$ And on a Dehn twist about a bounding simple closed curve $\gamma$ of genus $k,$ by (cf. \cite[Lemma 12a]{jo1})
$$\sigma(T_\gamma)= \sum^k_{i=1}\overline{c}_i\overline{d}_i,$$
where $\{c_i,d_i\}_i$ is the symplectic basis of the subsurface of genus $k$ bounded by $\gamma.$

By \cite[Lemma 13]{jo1}, the BCJ-homomorphism $\sigma$ is $\mathcal{M}_{g,1}$-equivariant. The action of $\mathcal{M}_{g,1}$ on $\mathfrak{B}_3$ factors through $Sp_{2g}(\mathbb{Z})$ via the symplectic representation $\Psi:\mathcal{M}_{g,1}\rightarrow Sp_{2g}(\mathbb{Z})$ where $Sp_{2g}(\mathbb{Z})$ acts on $\mathfrak{B}_3$ via its action on $H_1(\Sigma_{g,1},\mathbb{Z})$ modulo $2$. In other words, for a given $f\in \mathcal{M}_{g,1}$ and three elements $z_i,z_j,z_k$ of $H_1(\Sigma_{g,1},\mathbb{Z}/2),$ the action is given by:
$$f\cdot(\overline{z}_i\overline{z}_j\overline{z}_k)=\overline{\Psi(f)\cdot z_i}\;\overline{\Psi(f)\cdot z_j}\;\overline{\Psi(f)\cdot z_k}.$$

\begin{lemma}
\label{lem_B_3-inv}
Provided $g\geq 3,$ the isomorphisms
$(\mathfrak{B}_2)_{GL_g(\mathbb{Z})}\cong \langle \overline{1}, \overline{a}_1 \overline{b}_1\rangle \cong(\mathbb{Z}/2)^2$ hold, and provided $g\geq 4$ the isomorphism $(\mathfrak{B}_3)_{GL_g(\mathbb{Z})}\cong \langle \overline{1}\rangle=\mathbb{Z}/2$ also holds.
\end{lemma}

\begin{proof}
Since the action of $GL_g(\mathbb{Z}) \subset Sp_{2g}(\mathbb{Z})$ on the Boolean algebra $\mathfrak{B}_3$ does not increase the degree, we will compute coinvariants bottom up. Denote by $E_{ij}$ the matrix of size $g\times g$ with $1$ at position $(i,j)$ and zero elsewhere. Keep in mind that the symmetric group $\mathfrak{S}_g \subset GL_g(\mathbb{Z})$ acts on $H$ by permuting the indices of the generating set $\{a_i,b_i\}_{1 \leq i \leq g}$. 
 \vspace{0.2cm}

\textbf{In degree 0} we only have one element $\overline{1}$, and it is clearly invariant.

\textbf{In degree 1} the coinvariant module is generated by the elements $\overline{a}_1,\overline{b}_1$.
    
    Apply the element $G= Id_g + E_{21}$, to get:
    \[
    G \cdot \overline{a}_1 = \overline{a_1 +a_2}=\overline{a}_1 +\overline{a}_2.
    \]
    
    So in the coinvariants quotient $\overline{a}_2=0$, but this is the class of $\overline{a}_1$.
    
    The same computation shows that in the quotient $\overline{b}_1=0$.
    
\textbf{In degree 2} the coinvariants module is generated by the products
\[\overline{a}_1\overline{a}_2,\quad \overline{a}_1\overline{b}_2,\quad  \overline{a}_1\overline{b}_1, \quad\overline{b}_1\overline{b}_2.\] 
    
    We proceed now as before. \begin{itemize}
    	\item Acting by $G =Id_g + E_{32}$ on $\overline{a}_1\overline{a}_2$ gives:
    	\[
    	G\cdot \overline{a}_1\overline{a}_2 = \overline{a}_1 (\overline{a_2 +a_3}) = \overline{a}_1\overline{a}_2 + \overline{a}_1\overline{a}_3.
    	\]
    	So in the coinvariants quotient $\overline{a}_1\overline{a}_3=0,$ but this is the class of $\overline{a}_1 \overline{a}_2$.
    	
    	The same holds true for $\overline{b}_1\overline{b}_2$.
    	\item Acting by $G =Id_g + E_{21}$ on $\overline{a}_1\overline{b}_2$ gives:
    	\[
    	G \cdot \overline{a}_1\overline{b}_2 = (\overline{a_1 + a_2})(\overline{-b_1 +b_2}) = -\overline{a}_1\overline{b}_1 + \overline{a}_1\overline{b}_2 - \overline{a}_2\overline{b}_1 + \overline{a}_2\overline{b}_2.
    	\]
    	Since in the coinvariants quotient $\overline{a}_1\overline{b}_1=\overline{a}_2\overline{b}_2,$ we get that $\overline{a}_2\overline{b}_1 = 0$ in the coinvariants quotient. But this is the class of $\overline{a}_1\overline{b}_2.$
    \end{itemize}

    Thus we are left with $\overline{a}_1\overline{b}_1$. We show that this element is not zero in the coinvariants module:
    Consider the $GL_g(\mathbb{Z})$-invariant homomorphism $p:\mathfrak{B}_2\rightarrow \mathfrak{B}_2/\mathfrak{B}_1\xrightarrow{\overline{\omega}} \mathbb{Z}/2,$ where $\overline{\omega}$ sends an element $\overline{x}\overline{y}$ of $\mathfrak{B}_2/\mathfrak{B}_1$ to $\omega(x,y) \; \text{mod }2.$ Applying this homomorphism to $\overline{a}_1\overline{b}_1$ we get that $p(\overline{a}_1\overline{b}_1)=1\neq 0$ and therefore  $\overline{a}_1\overline{b}_1\neq 0.$
    
    All this shows that $(\mathfrak{B}_2)_{GL_g(\mathbb{Z})}\cong \langle \overline{1}, \overline{a}_1 \overline{b}_1\rangle \cong(\mathbb{Z}/2)^2.$
    
\textbf{In degree 3} the coinvariants module is generated by the products
$$\overline{a}_1\overline{a}_2\overline{a}_3,\quad \overline{a}_1\overline{a}_2\overline{b}_3,\quad \overline{a}_1\overline{a}_2\overline{b}_2,\quad \overline{a}_1\overline{b}_1\overline{b}_2,\quad \overline{a}_1\overline{b}_2\overline{b}_3,\quad \overline{b}_1\overline{b}_2\overline{b}_3.$$

Here is where we need  $g \geq 4$. 
    \begin{itemize}
    	\item Acting by $G = Id_g + E_{41}$ on $\overline{a}_1\overline{a}_2\overline{a}_3$ gives
    	\[
    	G \cdot \overline{a}_1\overline{a}_2\overline{a}_3 = (\overline{a_1 + a_4} )\overline{a}_2 \overline{a}_3 = \overline{a}_1\overline{a}_2\overline{a}_3 + \overline{a}_4\overline{a}_2\overline{a}_3.
    	\]
    	So in the coinvariants quotient $\overline{a}_4\overline{a}_2\overline{a}_3 = 0 = \overline{a}_1\overline{a}_2\overline{a}_3$. The action by the same element shows in the same way that in the coinvariants quotient $\overline{a}_1\overline{a}_2\overline{b}_3 =0$, $\overline{a}_1\overline{a}_2\overline{b}_2 =0$, and similarly $\overline{a}_1\overline{b}_1\overline{b}_2 =0 = \overline{a}_1\overline{b}_2\overline{b}_3 =\overline{b}_1\overline{b}_2\overline{b}_3$. 
    	\item Acting by  $G = Id_g + E_{21}$ on $\overline{a}_1\overline{a}_2\overline{b}_2$, whose transpose-inverse is $Id_g -E_{12}$, we have:
    	\[
    	G\cdot \overline{a}_1\overline{a}_2\overline{b}_2 = (\overline{a_1 + a_2})\overline{a}_2(\overline{b_2-b_1})= \overline{a}_1\overline{a}_2\overline{b}_2  - \overline{a}_1\overline{a}_2\overline{b}_1 + \overline{a}_2\overline{b}_2 -\overline{a}_2\overline{b}_1.
    	\]
    	So in the coinvariants quotient $- \overline{a}_1\overline{a}_2\overline{b}_1 + \overline{a}_2\overline{b}_2 -\overline{a}_2\overline{b}_1 =0$, but we already know that  in the quotient $ \overline{a}_1\overline{a}_2\overline{b}_1  = 0 = \overline{a}_2\overline{b}_1$, so finally $\overline{a}_2\overline{b}_2 = 0  = \overline{a}_1\overline{b}_1$.
    \end{itemize}	 
All this shows that $(\mathfrak{B}_3)_ {GL_g(\mathbb{Z})}$ is $\mathbb{Z}/2$, generated by $\overline{1}$.
\end{proof}

\subsection{Handlebodies}

Let $B_1,\ldots ,B_n$ be a collection of closed $3$-balls and $D_1,\ldots , D_m,$ $D'_1,\ldots ,D'_m$ be a collection of pairwise disjoint disks in $\bigcup \partial B_i.$ For each $1\leq i\leq m,$ consider a homeomorphism $\phi_i:D_i\rightarrow D'_i$. Denote $H$ the result of gluing along $\phi_1,$ then gluing along $\phi_2,$ and so on.
We say that $H$ is a \textit{handlebody} if after the final gluing if $H$ is connected. Equivalently, we say that a $3$-manifold with boundary $H$ is a handlebody if there exists a collection $\{D_1,\ldots ,D_m \}$ of properly embedded essential disks (i.e. properly embedded disk whose boundary does not bound a disk in the boundary of $H$) such that the complement of a regular neighbourhood of $\bigcup D_i$ is a collection of balls. Such family of disks is called a \textit{system of disks} of $H$ and we say that this system of disks is \textit{minimal} if its complement is connected.
The existence of such family of disks is ensured by \cite[Lemma~2.2]{jes}.

We now state a basic $3$-dimensional topology lemma that is a consequence of the loop theorem (cf. \cite[Theorem 3.1]{hatch}) and of which we will omit the proof.

\begin{lemma}
\label{lem_CBP_disks}
Let $D_\beta,$ $D_{\beta'}$ be two essential proper embedded disks in $\mathcal{H}_g$ with respective boundaries $\beta,$ $\beta'$ such that the union of $\beta$ and $\beta'$ bounds a subsurface of $\Sigma_{g,1}$ with two boundary components. There exist $g-1$ essential proper embedded disks $D_{\beta_1},\ldots ,D_{\beta_{g-1}}$ in $\mathcal{H}_g,$ with boundaries $\beta_1,\ldots ,\beta_{g-1}$ respectively, such that
$$Int(\mathcal{H}_g)-N\left( D_\beta \cup D_{\beta'} \cup D_{\beta_1} \cup \ldots \cup D_{\beta_{g-1}}\right)$$
is the disjoint union of two open $3$-balls.
\end{lemma}

\section{The Luft Group and CBP-twists}
\label{section_luft}
Denote by $\mathcal{L}_{g,1}$ the kernel of the map $\mathcal{B}_{g,1} \twoheadrightarrow Aut(\pi_1(\mathcal{H}_g))$, the \textit{Luft group,} which was identified by Luft in \cite{luft} as the ``twist group'' of the handlebody $\mathcal{H}_g$, and
by $\mathcal{LTB}_{g,1}$ the intersection $\mathcal{L}_{g,1} \cap \mathcal{TB}_{g,1}$, the Luft-Torelli group. In \cite{pitsch}, Pitsch characterized $\mathcal{LTB}_{g,1}$ as the group generated by contractible bounding pair twists (abbreviated CBP-twists). A CBP-twist is a map of the form $T_\beta T_{\beta'}^{-1},$ where $\beta$ and $\beta'$ are two homologous non-isotopic and disjoint simple closed curves on $\Sigma_{g,1}$ such that each one is not null-homologous in $H_1(\Sigma_{g,1};\mathbb{Z})$ and bounds a properly embedded disk in $\mathcal{H}_g.$ In all what follows we refer as CBP-twists of genus $k$ to those CBP-twists $T_{\beta}T_{\beta'}^{-1}$ such that the union of $\beta$ and $\beta$ bounds a subsurface of $\Sigma_{g,1}$ of genus $k$ with two boundary components.

Consider the short exact sequence given in \cite[Lemma 1]{pitsch},
\begin{equation*}
\xymatrix@C=7mm@R=10mm{1 \ar@{->}[r] & \mathcal{TB}_{g,1} \ar@{->}[r] & \mathcal{B}_{g,1} \ar@{->}[r]^-{\Psi} & GL_g(\mathbb{Z})\ltimes S_g(\mathbb{Z}) \ar@{->}[r] & 1 .}
\end{equation*}
Restricting this short exact sequence to the Luft group $\mathcal{L}_{g,1},$ we get that
\begin{proposition}
\label{prop_surj_Luft}
There is a short exact sequence
\begin{equation}
\label{ses_L}
\xymatrix@C=7mm@R=10mm{1 \ar@{->}[r] & \mathcal{LTB}_{g,1} \ar@{->}[r] & \mathcal{L}_{g,1} \ar@{->}[r]^-{\Psi} & S_g(\mathbb{Z}) \ar@{->}[r] & 1 .}
\end{equation}
\end{proposition}
 
\begin{proof}
Notice that $\mathcal{L}_{g,1}$ is contained in the kernel of $\mathcal{B}_{g,1} \twoheadrightarrow Aut(H_1(\mathcal{H}_g))=GL_g(\mathbb{Z}).$ Then $\Psi(\mathcal{L}_{g,1})\subset S_g(\mathbb{Z}).$ Now we prove that $\Psi:\mathcal{L}_{g,1}\rightarrow S_g(\mathbb{Z})$ is surjective. 
Recall that $S_g(\mathbb{Z})$ is generated by the family of matrices $\{E_{ii} \mid 1\leq i \leq g\}\cup\{SE_{ij} \mid 1\leq i<j \leq g\},$
where $E_{ij}$ denotes the matrix with $1$ at the position $(i,j)$ and $0$'s elsewhere, and $SE_{ij}=E_{ij}+E_{ji}$ for $i\neq j$.
Thus it is enough to find a preimage for each $E_{ii}$ and $SE_{ij}.$
Consider $\beta_i,$ $\beta_j$ and $\gamma_{ij}$ the curves given in the following figure:

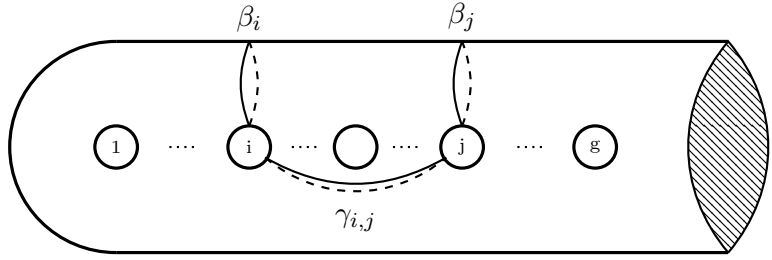
\begin{figure}[H]
\begin{center}
\begin{tikzpicture}[scale=.7]
\draw[very thick] (-4.5,-2) -- (7,-2);
\draw[very thick] (-4.5,2) -- (7,2);
\draw[very thick] (-4.5,2) arc [radius=2, start angle=90, end angle=270];
\draw[very thick] (-4.5,0) circle [radius=.4];
\draw[thick, dotted] (-3.5,0) -- (-3,0);
\draw[very thick] (-2,0) circle [radius=.4];
\draw[thick, dotted] (-1.2,0) -- (-0.7,0);
\draw[thick, dotted] (0.7,0) -- (1.2,0);
\draw[very thick] (2,0) circle [radius=.4];
\draw[thick, dotted] (3.5,0) -- (3,0);
\draw[very thick] (4.5,0) circle [radius=.4];

\draw[very thick] (0,0) circle [radius=.4];

\draw[thick] (-1.7,-0.2) to [out=-30,in=210] (1.7,-0.2);
\draw[thick, dashed] (-1.7,-0.2) to [out=-40,in=220] (1.7,-0.2);
\node [below] at (0,-1) {$\gamma_{i,j}$};

\draw[thick,dashed] (-2,0.4) to [out=70,in=-70] (-2,2);
\draw[thick] (-2,0.4) to [out=110,in=-110] (-2,2);

\draw[thick,dashed] (2,0.4) to [out=70,in=-70] (2,2);
\draw[thick] (2,0.4) to [out=110,in=-110] (2,2);

\node[above] at (-2,2) {$\beta_i$};
\node[above] at (2,2) {$\beta_j$};

\node at (-4.5,0) {\tiny{1}};
\node at (-2,0) {\tiny{i}};
\node at (2,0) {\tiny{j}};
\node at (4.5,0) {\tiny{g}};

\draw[thick,pattern=north west lines] (7,-2) to [out=130,in=-130] (7,2) to [out=-50,in=50] (7,-2);

\end{tikzpicture}
\end{center}
\caption{Curves involved in the lift of $SE_{ij}$ to $\mathcal{L}_{g,1}$.}
\end{figure}

Notice that the curves $\beta_k,$ $\gamma_{ij}$ are contractible in $\mathcal{H}_g.$ Then $T_{\beta_k},$ $T_{\beta_i}^{-1}T_{\gamma_{ij}}T_{\beta_j}^{-1}$ are elements of the Luft group $\mathcal{L}_{g,1}$ with the following images through the symplectic representation:
$$\Psi(T_{\beta_k}^{-1})=\left(\begin{matrix}
Id & 0 \\
E_{kk} & Id
\end{matrix}\right), \qquad
\Psi(T_{\beta_i}^{-1}T_{\gamma_{ij}}T_{\beta_j}^{-1})=\left(\begin{matrix}
Id & 0 \\
SE_{ij} & Id
\end{matrix}\right).$$
\end{proof}

\begin{proposition}
\label{prop_prod_B}
For every $h\in \mathcal{B}_{g,1},$ there exist elements $l\in \mathcal{L}_{g,1},$ $f\in \mathcal{AB}_{g,1}$ and $\xi_b\in \mathcal{TB}_{g,1}$ such that
$h=\xi_b f l,$ i.e.
$$\mathcal{B}_{g,1}=\mathcal{TB}_{g,1}\cdot \mathcal{AB}_{g,1} \cdot\mathcal{L}_{g,1}.$$
\end{proposition}

\begin{proof}
Consider the short exact sequence
\begin{equation}
\label{ses_B}
\xymatrix@C=7mm@R=10mm{1 \ar@{->}[r] & \mathcal{TB}_{g,1} \ar@{->}[r] & \mathcal{B}_{g,1} \ar@{->}[r]^-{\Psi} & GL_g(\mathbb{Z})\ltimes S_g(\mathbb{Z}) \ar@{->}[r] & 1 }
.\end{equation}
By \cite[Lemma 3]{pitsch} and Proposition \ref{prop_surj_Luft} we know that $\Psi(\mathcal{AB}_{g,1})\cong GL_g(\mathbb{Z})$ and
$\Psi(\mathcal{L}_{g,1})\cong S_g(\mathbb{Z}).$
Then, given $h\in \mathcal{B}_{g,1}$ there exist $f\in \mathcal{AB}_{g,1},$ $l\in \mathcal{L}_{g,1}$ such that
$\Psi(h)=\Psi(fl),$
and by the short exact sequence \eqref{ses_B}, there exists an element $\xi_b\in \mathcal{TB}_{g,1}$ such that $h=\xi_b f l.$

\end{proof}

\begin{proposition}
\label{prop_trans_B}
The group $\mathcal{B}_{g,1}$ acts transitively on CBP-twists of a given genus.
\end{proposition}

\begin{proof}
Let $T_\zeta T^{-1}_{\zeta'},$ $T_\beta T_{\beta'}^{-1}$ be CBP-twists of genus $k$ on $\Sigma_{g,1}.$ We prove that there exists $\psi\in\mathcal{B}_{g,1}$ such that $\psi(\beta)=\zeta,$ $\psi(\beta')=\zeta'$ getting that
$$\psi(T_\beta T_{\beta'}^{-1})\psi^{-1}=T_{\psi(\beta)} T_{\psi(\beta')}^{-1}=T_\zeta T_{\zeta'}^{-1}.$$

Since $T_\beta T_{\beta'}^{-1},$ $T_\zeta T_{\zeta'}^{-1}$ are CBP-twists of genus $k,$ there exist properly embedded disks $D_{\beta},D_{\beta'}, D_{\zeta},D_{\zeta'}$ in $\mathcal{H}_g$ with respective boundaries $\beta, \beta', \zeta, \zeta'.$ Then,
by Lemma \ref{lem_CBP_disks}, there exist $g-1$ essential proper embedded disks $D_{\beta_1},\ldots ,D_{\beta_{g-1}}$ (resp. $D_{\zeta_1},\ldots ,D_{\zeta_{g-1}}$) in $\mathcal{H}_g,$ with boundaries $\beta_1,\ldots ,\beta_{g-1}$ (resp. $\zeta_1,\ldots ,\zeta_{g-1}$), such that
\begin{align*}
& Int(\mathcal{H}_g)-N\left( D_\beta \cup D_{\beta'} \cup D_{\beta_1} \cup \ldots \cup D_{\beta_{g-1}}\right) \\
\big(\text{resp.}\quad & Int(\mathcal{H}_g)-N\left( D_\zeta \cup D_{\zeta'} \cup D_{\zeta_1} \cup \ldots \cup D_{\zeta_{g-1}}\right)\;\big)
\end{align*}
is the disjoint union of two open $3$-balls.
 
Since $T_\zeta T^{-1}_{\zeta'},$ $T_\beta T_{\beta'}^{-1}$ are BP-maps of the same genus, by the change of coordinates principle (cf. \cite[Section 1.3]{farb}),
there is a homeomorphism $\phi$ from $\Sigma_{g,1}$ to $\Sigma_{g,1}$ sending $\{\beta, \beta', \beta_1 , \ldots , \beta_{g-1}\}$ to $\{\zeta, \zeta', \zeta_1 , \ldots , \zeta_{g-1}\}$ respectively.
By \cite[Lemma 2.9]{jes}, $\phi$ extends to a homeomorphism $\psi$ on $\mathcal{H}_g$ and therefore $\phi\in \mathcal{B}_{g,1}.$
\end{proof}

\begin{proposition}
\label{prop_CBP_1}
Every CBP-twist of genus $k$ is a product of $k$ CBP-twists of genus $1.$
\end{proposition}

\begin{proof}
Let $T_\beta T_{\beta'}^{-1}$ be a CBP-twist of genus $k.$ Consider the following simple closed curves in the standarly embedded surface $\Sigma_{g,1}:$

\begin{figure}[H]
\begin{center}
\begin{tikzpicture}[scale=.7]
\draw[very thick] (-12,-2) -- (3,-2);
\draw[very thick] (-12,2) -- (3,2);
\draw[thick, dotted] (-7.5,0) -- (-6.5,0);
\draw[very thick] (-12,2) arc [radius=2, start angle=90, end angle=270];

\draw[very thick] (-12,0) circle [radius=.4];
\draw[very thick] (-9.5,0) circle [radius=.4];
\draw[very thick] (-4.5,0) circle [radius=.4];
\draw[very thick] (-2,0) circle [radius=.4];
\draw[very thick] (1,0) circle [radius=.4];

\draw[thick,dashed] (-2,0.4) to [out=70,in=-70] (-2,2);
\draw[thick] (-2,0.4) to [out=110,in=-110] (-2,2);

\draw[thick,dashed] (-2,-0.4) to [out=-70,in=70] (-2,-2);
\draw[thick] (-2,-0.4) to [out=-110,in=110] (-2,-2);

\draw[thick] (-2.4,0) to [out=90,in=0] (-4.5,1.6) to [out=180,in=0] (-9.5,1.6) to [out=180,in=90] (-11,0) to [out=-90,in=110] (-10.8,-2);
\draw[thick, dashed] (-2.4,0) to [out=90,in=0] (-4.5,1.2) to [out=180,in=0] (-9.5,1.2) to [out=180,in=90] (-10.6,0) to [out=-90,in=70] (-10.8,-2);

\draw[thick] (-2.4,0) to [out=90,in=0] (-4.5,1.3) to [out=180,in=0] (-7,1.3) to [out=180,in=90] (-8.5,0) to [out=-90,in=110] (-8.3,-2);
\draw[thick, dashed] (-2.4,0) to [out=90,in=0] (-4.5,0.9) to [out=180,in=0] (-7,0.9) to [out=180,in=90] (-8.1,0) to [out=-90,in=70] (-8.3,-2);

\draw[thick] (-2.4,0) to [out=90,in=0] (-4.5,1) to [out=180,in=90] (-6,0) to [out=-90,in=110] (-5.8,-2);
\draw[thick, dashed] (-2.4,0) to [out=90,in=0] (-4.5,0.6) to [out=180,in=90] (-5.6,0) to [out=-90,in=70] (-5.8,-2);

\node [above] at (-2,2) {$\zeta$};
\node [below] at (-2,-2) {$\zeta'$};
\node [below] at (-10.8,-2) {$\zeta_1$};
\node [below] at (-8.3,-2) {$\zeta_2$};
\node [below] at (-5.8,-2) {$\zeta_{k-1}$};

\draw[thick, dotted] (-7.5,-2.6) -- (-7,-2.6);

\node at (-12,0) {\tiny{1}};
\node at (-9.5,0) {\tiny{2}};
\node at (-2,0) {\tiny{k}};
\node at (-4.5,0) {\tiny{k-1}};
\node at (1,0) {\tiny{g}};

\draw[thick, dotted] (-1,0) -- (0,0);

\draw[thick,pattern=north west lines] (3,-2) to [out=130,in=-130] (3,2) to [out=-50,in=50] (3,-2);
\end{tikzpicture}
\end{center}
\caption{CBP-twist of genus $k$ as a product of $k$ CBP-twist of genus $1$.}
\end{figure}
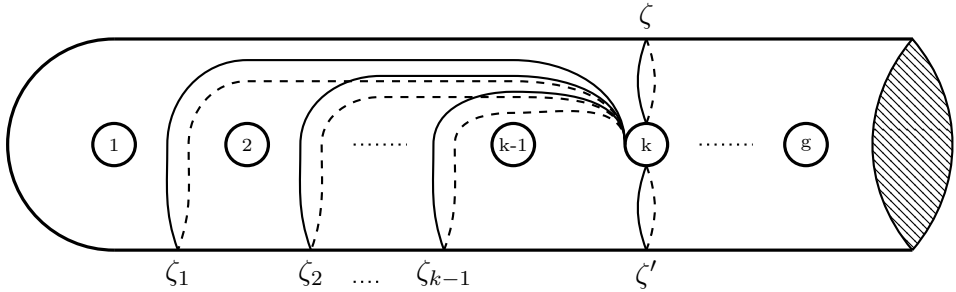

Observe that $T_\zeta T_{\zeta'}^{-1}$ is a CBP-twist of genus $k$ and that for $i=0,\ldots k-1,$ the maps $T_{\zeta_i} T_{\zeta_{i+1}}^{-1}$ are CBP-twists of genus $1,$  where $\zeta_0=\zeta,$ $\zeta_k=\zeta'.$
By Proposition~\ref{prop_trans_B},
there exists an element $h\in \mathcal{B}_{g,1}$ such that
$T_\beta T_{\beta'}^{-1}=hT_\zeta T_{\zeta'}^{-1}h^{-1}.$
Therefore,
\begin{align*}
T_\beta T_{\beta'}^{-1}=hT_\zeta T_{\zeta'}^{-1}h^{-1}= &
(hT_{\zeta_0} T_{\zeta_1}^{-1}h^{-1})(hT_{\zeta_1} T_{\zeta_2}^{-1}h^{-1})\cdots
(hT_{\zeta_{k-1}} T_{\zeta_k}^{-1}h^{-1})= \\
= &
(T_{h(\zeta_0)} T_{h(\zeta_1)}^{-1})(T_{h(\zeta_1)} T_{h(\zeta_2)}^{-1})\cdots
(T_{h(\zeta_{k-1})} T_{h(\zeta_k)}^{-1}).
\end{align*}
Since $T_{\zeta_i}T_{\zeta_{i+1}}^{-1}$ is a CBP-twists of genus $1$ for $i=0,\ldots k-1,$ and $h\in \mathcal{B}_{g,1},$ the element
$T_{h(\zeta_i)} T_{h(\zeta_{i+1})}^{-1}$ is also a CBP-twists of genus $1$ for $i=0,\ldots k-1.$
\end{proof}

\begin{remark}
A posteriori we found that Proposition \ref{prop_CBP_1} was obtained independently by Omori in \cite{omori}.
\end{remark}

\section{Invariants of Integral Homology $3$-Spheres}
Throughout this section we set $A$ an abelian group and $A_2$ the subgroup of $2$-torsion elements of $A.$
Consider an $A$-valued invariant of homology $3$-spheres $F: \mathcal{S}^3 \rightarrow A.$
By \cite[Theorem 1]{pitsch} there is a bijection
\begin{equation}
\label{bij_S3}
\lim_{g\to \infty}(\mathcal{TA}_{g,1}\backslash\mathcal{T}_{g,1}/\mathcal{TB}_{g,1})_{\mathcal{AB}_{g,1}} \xrightarrow{\sim}  \mathcal{S}^3.
\end{equation}
Precomposing an invariant $F$ with the canonical maps $\mathcal{T}_{g,1}\rightarrow \lim_{g\rightarrow \infty} \mathcal{T}_{g,1}/\sim \rightarrow  \mathcal{S}^3$ we get a family of maps $\{F_g\}_g$ with $F_g:\mathcal{T}_{g,1}\rightarrow A$ satisfying the following properties:
\begin{enumerate}[i)]
\item $F_{g+1}(x)=F_g(x) \quad \text{for every }x\in \mathcal{T}_{g,1},$
\item $F_g(\xi_a x\xi_b)=F_g(x) \quad \text{for every } x\in \mathcal{T}_{g,1},\;\xi_a\in \mathcal{TA}_{g,1},\;\xi_b\in \mathcal{TB}_{g,1},$
\item $F_g(\phi x \phi^{-1})=F_g(x)  \quad \text{for every }  x\in \mathcal{T}_{g,1}, \; \phi\in \mathcal{AB}_{g,1}.$
\end{enumerate}
Because of property i), without loss of generality we can assume $g \geq 4,$ this avoids
having to deal with some peculiarities in the homology of low genus mapping class groups.

Using this framework we prove that modulo a multiplicative constant $x\in A_2$ there is only one family of homomorphisms $\{F_g\}_g$ satisfying the aforementioned properties and that this family reassembles to the Rohlin invariant.
For such purpose we have to understand the following three groups:
$$H_1(\mathcal{T}_{g,1};\mathbb{Z})_{\mathcal{AB}_{g,1}}, \qquad H_1(\mathcal{TA}_{g,1};\mathbb{Z})_{\mathcal{AB}_{g,1}}, \qquad H_1(\mathcal{TB}_{g,1};\mathbb{Z})_{\mathcal{AB}_{g,1}}.$$

\begin{proposition}
\label{prop-Torelli-GL}
For a given integer $g\geq 4,$ the Birman Craggs Johnson homomorphism $\sigma: \mathcal{T}_{g,1}\rightarrow \mathfrak{B}_3$
composed with the projection $\mathfrak{B}_3\rightarrow\mathfrak{B}_0$ induces an isomorphism $$H_1(\mathcal{T}_{g,1};\mathbb{Z})_{\mathcal{AB}_{g,1}}\cong \mathbb{Z}/2.$$
\end{proposition}

\begin{proof}
By the fundamental result of Johnson~\cite{jon_3}, we have an extension:
	\[
	\xymatrix{
0 \ar[r] & \mathfrak{B}_2   \ar[r] & H_1(\mathcal{T}_{g,1};\mathbb{Z}) \ar[r] & \Lambda^3 H \ar[r] & 0.	
}
	\]
	
%	Moreover the Birman-Craggs homomorphisms give us a commutative triangle:
%	\[
%	\xymatrix{
%	\mathfrak{B}_g^2 \ar@{_{(}->}[dr]   \ar[r] & H_1(\mathcal{T}_{g,1};\mathbb{Z}) \ar[d]^{BC} \\
%	 & \mathfrak{B}^3_g
%}
%	\]
	All the maps appeared in the short exact sequence above are $\mathcal{M}_{g,1}$-equivariant, where the action of $\mathcal{M}_{g,1}$ on the above three groups is through the symplectic action on $H$. In particular, taking $\mathcal{AB}_{g,1}$-coinvariants we get an exact sequence:
	\begin{equation}
	\label{ses_BC}
	\xymatrix{
		 (\mathfrak{B}_2)_{GL_g(\mathbb{Z})} \ar[r] & H_1(\mathcal{T}_{g,1};\mathbb{Z})_{GL_g(\mathbb{Z})} \ar[r] & (\Lambda^3 H)_{GL_g(\mathbb{Z})} \ar[r] & 0.
		}	
	\end{equation}

	First, observe that $-Id \in GL_g(\mathbb{Z})$ acts by multiplication by $-1$ on $H$ and hence on $\Lambda^3 H$. Therefore, for any $w \in (\Lambda^3 H)_{GL_g(\mathbb{Z})}$, $-w = w$ and hence there are isomorphisms
	$$(\Lambda^3H) _{GL_g(\mathbb{Z})}\simeq (\Lambda^3H_2) _{GL_g(\mathbb{Z})}\simeq (\mathfrak{B}_3/\mathfrak{B}_2)_{GL_g(\mathbb{Z})}.$$
	By Lemma \ref{lem_B_3-inv} this last group is zero, and by the exact sequence \eqref{ses_BC} the $\mathbb{Z}/2$-vector space $(\mathfrak{B}_2)_{GL_g(\mathbb{Z})}$ surjects onto $H_1(\mathcal{T}_{g,1};\mathbb{Z})_{GL_g(\mathbb{Z})}$ getting that all elements of this last group have $2$-torsion.
Therefore by \cite[Theorem 1]{jon_3} the BCJ homomorphism $\sigma: \mathcal{T}_{g,1}\rightarrow \mathfrak{B}_3$ induces an isomorphism $H_1(\mathcal{T}_{g,1};\mathbb{Z})_{GL_g(\mathbb{Z})}\cong (\mathfrak{B}_3)_{GL_g(\mathbb{Z})}$ and we conclude by Lemma \ref{lem_B_3-inv}.
\end{proof}

\begin{lemma}
\label{lem_TB,TA}
Provided $g\geq 4,$ the groups 
$H_1 (\mathcal{TB}_{g,1};\mathbb{Z})_{\mathcal{AB}_{g,1}}$ and $H_1 (\mathcal{TA}_{g,1};\mathbb{Z})_{\mathcal{AB}_{g,1}}$
are zero.
\end{lemma}

\begin{proof}
We only give the proof for $\mathcal{TB}_{g,1}$ the other case is similar.

Denote by $IA$ the kernel of the natural map $Aut(\pi_1(\mathcal{H}_g))\rightarrow GL_g(\mathbb{Z})$.
Consider the following short exact sequence:
\begin{equation}
\label{seq_LTB}
\xymatrix@C=7mm@R=7mm{ 1 \ar@{->}[r] & \mathcal{LTB}_{g,1} \ar@{->}[r] & \mathcal{TB}_{g,1} \ar@{->}[r]  & IA \ar@{->}[r] & 1 .}
\end{equation}
Taking $\mathcal{AB}_{g,1}$-coinvariants on the associated 5-term exact sequence, we get another exact sequence
\begin{equation*}
\xymatrix@C=7mm@R=7mm{ H_1(\mathcal{LTB}_{g,1};\mathbb{Z})_{\mathcal{TB}_{g,1}\cdot \mathcal{AB}_{g,1}} \ar@{->}[r] & H_1(\mathcal{TB}_{g,1};\mathbb{Z})_{\mathcal{AB}_{g,1}} \ar@{->}[r]  & H_1(IA;\mathbb{Z})_{\mathcal{AB}_{g,1}} \ar@{->}[r] & 0 ,}
\end{equation*}
and we conclude by Lemma \ref{lema-IA-zero} and Lemma \ref{lema-LTB-zero}.
\end{proof}

\begin{lemma}
\label{lema-IA-zero}
For a given integer $g\geq 3,$ the group $(H_1(IA;\mathbb{Z}))_{\mathcal{AB}_{g,1}}$ is zero.
\end{lemma}

\begin{proof}
By \cite[Corollary 2.1]{luft} the action of $\mathcal{B}_{g,1}$ on the fundamental group of the inner handlebody $\mathcal{H}_g$ induces a surjective map $\mathcal{B}_{g,1} \twoheadrightarrow Aut (\pi_1(\mathcal{H}_g)).$ Indeed, the restriction of this map to $\mathcal{AB}_{g,1}$ also gives
a surjective map $\mathcal{AB}_{g,1} \twoheadrightarrow Aut (\pi_1(\mathcal{H}_g))$ (cf. the paragraph after Lemma 2 in \cite{pitsch}). Therefore we have an isomorphism
$$(H_1(IA;\mathbb{Z}))_{\mathcal{AB}_{g,1}}\cong (H_1(IA;\mathbb{Z}))_{Aut (\pi_1(\mathcal{H}_g))}.$$
According to Magnus \cite{magnus2}, for the given generators $\alpha_1, \cdots ,\alpha_g$ of $\pi_1(\mathcal{H}_g),$ the group $IA$ is normally generated as a subgroup of $Aut (\pi_1(\mathcal{H}_g))$ by the automorphism $K_{12}$ given by $K_{12}(\alpha_1)=\alpha_2\alpha_1\alpha_2^{-1}$ and $K_{12}(\alpha_i)=\alpha_i$ for $i\geq 2$. Then it is enough to show that $K_{12}$ is equivalent to zero.

Consider  $f\in Aut(\pi_1(\mathcal{H}_g))$ given by
$f(\alpha_3)=\alpha_3\alpha_2$ and $f(\alpha_i)=\alpha_i$ for $i\neq 3,$
with inverse $f^{-1}\in Aut(\pi_1(\mathcal{H}_g))$ given by
$f^{-1}(\alpha_3)=\alpha_3\alpha_2^{-1}$ and $f^{-1}(\alpha_i)=\alpha_i$ for $i\neq 3,$
and take the element $K_{13}\in IA$ given by
$K_{13}(\alpha_1)=\alpha_3\alpha_1\alpha_3^{-1}$ and $K_{13}(\alpha_i)=\alpha_i$ for $i\geq 2.$
Observe that
$$fK_{13}f^{-1}(\alpha_1)=\alpha_3\alpha_2\alpha_1\alpha_2^{-1}\alpha_3^{-1}\quad \text{and} \quad fK_{13}f^{-1}(\alpha_i)=\alpha_i\quad \text{for }i\geq 2.$$
Consequently, $fK_{13}f^{-1}=K_{12}K_{13}$
and the following equation holds
$$K_{13}=fK_{13}f^{-1}=K_{12}K_{13}=K_{12}+K_{13}.$$
Therefore $K_{12}$ is equivalent to zero.
\end{proof}

\begin{lemma}
\label{lema-LTB-zero}
For a given integer $g\geq 4,$ the group $(H_1(\mathcal{LTB}_{g,1};\mathbb{Z}))_{\mathcal{TB}_{g,1}\cdot \mathcal{AB}_{g,1}}$ is zero.
\end{lemma}

\begin{proof}
By Proposition \ref{prop_CBP_1}, the group
$\mathcal{LTB}_{g,1}$ is generated by CBP-twists of genus $1.$
Then it is enough to show that all CBP-twists of genus $1$ are equivalent to zero.
Next we divide the proof in two steps. In the first step we show that all CBP-twists of genus $1$ are equivalent, and in the second step we show that all CBP-twists of genus $1$ are equivalent to zero.

\;

\textbf{Step 1.} We show that all CBP-twists of genus $1$ are equivalent.
Consider the CBP-twist $T_\beta T^{-1}_{\beta'}\in \mathcal{LTB}_{g,1}$ of genus $1$ depicted in the following figure:
\begin{figure}[H]
\begin{center}
\begin{tikzpicture}[scale=.5]
\draw[very thick] (-7,-2) -- (3,-2);
\draw[very thick] (-7,2) -- (3,2);
\draw[very thick] (-7,2) arc [radius=2, start angle=90, end angle=270];
\draw[very thick] (-7,0) circle [radius=.4];
\draw[very thick] (-4.5,0) circle [radius=.4];
\draw[very thick] (-2,0) circle [radius=.4];
\draw[very thick] (1,0) circle [radius=.4];
\draw[thick, dotted] (-1,0) -- (0,0);

\draw[thick,dashed] (-4.5,0.4) to [out=70,in=-70] (-4.5,2);
\draw[thick] (-4.5,0.4) to [out=110,in=-110] (-4.5,2);

\draw[thick,dashed] (-4.5,-0.4) to [out=-70,in=70] (-4.5,-2);
\draw[thick] (-4.5,-0.4) to [out=-110,in=110] (-4.5,-2);

\node [above] at (-4.5,2) {$\beta$};
\node [below] at (-4.5,-2) {$\beta'$};

\draw[thick,pattern=north west lines] (3,-2) to [out=130,in=-130] (3,2) to [out=-50,in=50] (3,-2);
\node at (-7,0) {\tiny{1}};
\node at (-4.5,0) {\tiny{2}};
\node at (-2,0) {\tiny{3}};
\node at (1,0) {\tiny{g}};
\end{tikzpicture}
\end{center}
\caption{A contractible bounding pair of genus $1$ in $\Sigma_{g,1}$.}
\end{figure}
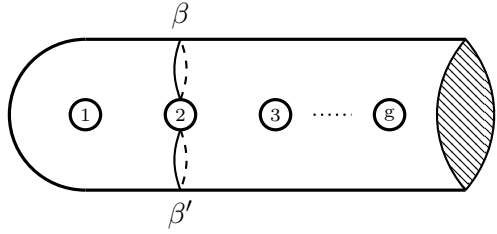

By Proposition \ref{prop_trans_B}, for every CBP-twist of genus $1,$ $T_\nu T_{\nu'}^{-1},$ on $ \Sigma_{g,1}$ there exists an element $h\in\mathcal{B}_{g,1}$ such that
$T_\nu T_{\nu'}^{-1}=hT_{\beta} T_{\beta'}^{-1}h^{-1},$
and by Proposition \ref{prop_prod_B}, there exist elements $l\in \mathcal{L}_{g,1},$ $f\in \mathcal{AB}_{g,1}$ and $\xi_b\in \mathcal{TB}_{g,1}$ such that
$h=\xi_b f l.$

Therefore in the coinvariant module we get that
\begin{align}
\label{eq_CBP_l}
T_\nu T_{\nu'}^{-1}=hT_{\beta} T_{\beta'}^{-1}h^{-1}=\xi_b f lT_{\beta} T_{\beta'}^{-1}l^{-1}f^{-1}\xi_b^{-1}= lT_{\beta} T_{\beta'}^{-1}l^{-1}.
\end{align}
Take the set-theoretic cross-section $s:\; S_g(\mathbb{Z})\rightarrow \mathcal{L}_{g,1}$ of $\Psi: \mathcal{L}_{g,1}\rightarrow S_g(\mathbb{Z})$, i.e. a function $s:\; S_g(\mathbb{Z})\rightarrow \mathcal{L}_{g,1}$ such that $\Psi\circ s=id$ not necessarily being an isomorphism, given by
$$
s(E_{ii})= T_{\beta_i}^{-1}, \qquad
s(SE_{ij})= \left\{\begin{array}{cc}
T_{\beta_1}^{-1}T_{\gamma_{1j}}T_{\beta_j}^{-1} & \text{for}\quad i=1 \\
T_{\beta_i}^{-1}T_{\gamma'_{ij}}T_{\beta_j}^{-1} & \text{for}\quad i\geq 2,
\end{array} \right. $$
where the curves $\beta_i,$ $\gamma_{ij},$ $\gamma'_{ij}$ are given in the following picture:
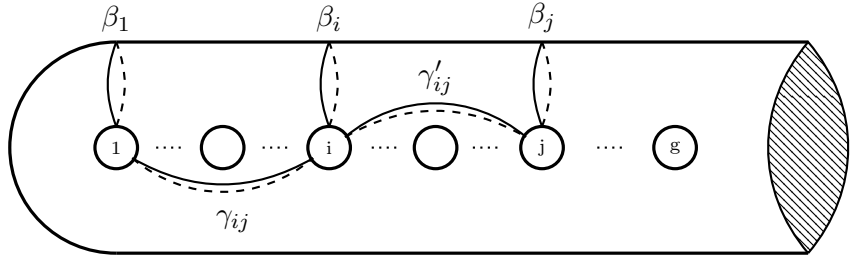
\begin{figure}[H]
\begin{center}
\begin{tikzpicture}[scale=.7]
\draw[very thick] (-6,-2) -- (7,-2);
\draw[very thick] (-6,2) -- (7,2);
\draw[very thick] (-6,2) arc [radius=2, start angle=90, end angle=270];
\draw[very thick] (-6,0) circle [radius=.4];
\draw[thick, dotted] (-5.25,0) -- (-4.75,0);
\draw[very thick] (-4,0) circle [radius=.4];
\draw[thick, dotted] (-3.25,0) -- (-2.75,0);
\draw[very thick] (-2,0) circle [radius=.4];
\draw[thick, dotted] (-1.2,0) -- (-0.7,0);
\draw[thick, dotted] (0.7,0) -- (1.2,0);
\draw[very thick] (2,0) circle [radius=.4];
\draw[thick, dotted] (3.5,0) -- (3,0);
\draw[very thick] (4.5,0) circle [radius=.4];
\draw[very thick] (0,0) circle [radius=.4];

\draw[thick] (-5.7,-0.2) to [out=-30,in=210] (-2.3,-0.2);
\draw[thick, dashed] (-5.7,-0.2) to [out=-40,in=220] (-2.3,-0.2);
\node [below] at (-3.8,-1) {$\gamma_{ij}$};

\draw[thick, dashed] (-1.7,0.2) to [out=30,in=-210] (1.7,0.2);
\draw[thick] (-1.7,0.2) to [out=40,in=-220] (1.7,0.2);
\node [above] at (0,0.8) {$\gamma'_{ij}$};

\draw[thick,dashed] (-6,0.4) to [out=70,in=-70] (-6,2);
\draw[thick] (-6,0.4) to [out=110,in=-110] (-6,2);

\draw[thick,dashed] (-2,0.4) to [out=70,in=-70] (-2,2);
\draw[thick] (-2,0.4) to [out=110,in=-110] (-2,2);

\draw[thick,dashed] (2,0.4) to [out=70,in=-70] (2,2);
\draw[thick] (2,0.4) to [out=110,in=-110] (2,2);

\node[above] at (-6,2) {$\beta_1$};
\node[above] at (-2,2) {$\beta_i$};
\node[above] at (2,2) {$\beta_j$};

\node at (-6,0) {\tiny{1}};
\node at (-2,0) {\tiny{i}};
\node at (2,0) {\tiny{j}};
\node at (4.5,0) {\tiny{g}};

\draw[thick,pattern=north west lines] (7,-2) to [out=130,in=-130] (7,2) to [out=-50,in=50] (7,-2);

\end{tikzpicture}
\end{center}
\caption{Curves involved in the set-theoretic cross-section $s$.}
\end{figure}

By the short exact sequence \eqref{ses_L}, given an element $l\in \mathcal{L}_{g,1},$ there exists an element $\xi_b\in \mathcal{LTB}_{g,1}$ such that $l=\xi_b s(\Psi(l)).$ Then, by Eq. \eqref{eq_CBP_l}, in the coinvariant module, we have that
\begin{equation}
\label{eq_CBP_s}
\begin{aligned}[b]
T_\nu T_{\nu'}^{-1}= \; &  lT_{\beta} T_{\beta'}^{-1}l^{-1}=\xi_b s(\Psi(l))T_{\beta} T_{\beta'}^{-1}s(\Psi(l))^{-1}\xi_b^{-1} \\
= \; & s(\Psi(l))T_{\beta} T_{\beta'}^{-1}s(\Psi(l))^{-1}.
\end{aligned}
\end{equation}
Now observe that $s(\Psi(l))$ is a product of the following elements:
$$\{T_{\gamma_{1i}}\mid \; i\geq 2\},\quad \{T_{\gamma'_{ij}}\mid \; 2\leq i<j \},\quad \{T_{\beta_i}\mid 1\leq i \leq g\}.$$
Since the curves
\begin{equation}
\label{died_curves}
\gamma_{12},\quad \{\gamma'_{ij}\mid \; 2\leq i<j \},\quad \{\beta_i\mid 1\leq i \leq g\},
\end{equation}
are disjoint with the curves
\begin{equation}
\label{survived_curves}
\beta,\quad \beta', \quad\{\gamma_{1j}\mid \; j\geq 3\},
\end{equation}
the geometric intersection number between a curve of the family \eqref{died_curves} and a curve of the family \eqref{survived_curves}
is zero. Therefore, the elements of the family of Dehn twists
\begin{equation}
\label{died_twists}
T_{\gamma_{12}},\quad \{T_{\gamma'_{ij}}\mid \; 2\leq i<j \},\quad \{T_{\beta_i}\mid 1\leq i \leq g\},
\end{equation}
commutes with the elements of the family of Dehn twists
\begin{equation}
\label{survived_twists}
T_{\beta},\quad T_{\beta'}, \quad\{T_{\gamma_{1j}}\mid \; j\geq 3\}.
\end{equation}
Furthermore, the elements of the family $\{T_{\gamma_{1j}}\mid \; j\geq 3\}$ commute between them because the curves of the family
$\{\gamma_{1j}\mid \; j\geq 3\}$ are pairwise disjoint.
Therefore,
\begin{equation}
\label{eq_CBP_expand}
s(\Psi(l))T_{\beta}  T_{\beta'}^{-1}s(\Psi(l))^{-1}
=(T_{\gamma_{13}}^{x_3}\cdots T_{\gamma_{1g}}^{x_g})T_{\beta}  T_{\beta'}^{-1}(T_{\gamma_{13}}^{x_3}\cdots T_{\gamma_{1g}}^{x_g})^{-1},
\end{equation}
for some $x_3, \ldots , x_g\in \mathbb{Z}.$
And as a consequence of the Eqs. \eqref{eq_CBP_s} and \eqref{eq_CBP_expand} we get the following equation:
\begin{equation}
\label{eq_CBP_product}
T_\nu T_{\nu'}^{-1}=(T_{\gamma_{13}}^{x_3}\cdots T_{\gamma_{1g}}^{x_g})T_{\beta} T_{\beta'}^{-1}(T_{\gamma_{13}}^{x_3}\cdots T_{\gamma_{1g}}^{x_g})^{-1}=T_{\beta} T_{(T_{\gamma_{13}}^{x_3}\cdots T_{\gamma_{1g}}^{x_g})(\beta')}^{-1}.
\end{equation}

Next, we prove that in the coinvariant module,
$$T_{\beta} T_{(T_{\gamma_{13}}^{x_3}\cdots T_{\gamma_{1g}}^{x_g})(\beta')}^{-1}= T_{\beta} T_{\beta'}^{-1}+\sum^g_{i=3}x_i\big(T_{\beta} T_{\beta'}^{-1}-T_{\beta} T_{T^{-1}_{\gamma_{1i}}(\beta')}^{-1}\big).$$
Consider the curves $\{\gamma_{1j}, \gamma'_{1j} \mid \; 3\leq j\geq g\}$ given in the following picture:

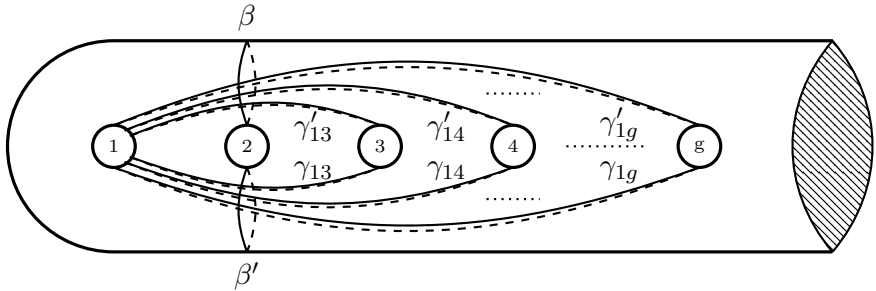
\begin{figure}[H]
\begin{center}
\begin{tikzpicture}[scale=.7]
\draw[very thick] (-2.5,-2) -- (11,-2);
\draw[very thick] (-2.5,2) -- (11,2);
\draw[very thick] (-2.5,2) arc [radius=2, start angle=90, end angle=270];

\draw[very thick] (0,0) circle [radius=.4];
\draw[very thick] (-2.5,0) circle [radius=.4];
\draw[very thick] (2.5,0) circle [radius=.4];
\draw[very thick] (5,0) circle [radius=.4];
\draw[very thick] (8.5,0) circle [radius=.4];

\draw[thick, dotted] (6,0) -- (7.5,0);
\draw[thick, dotted] (4.5,-1) -- (5.5,-1);
\draw[thick, dotted] (4.5,1) -- (5.5,1);

\draw[thick,dashed] (0,0.4) to [out=70,in=-70] (0,2);
\draw[thick] (0,0.4) to [out=110,in=-110] (0,2);

\draw[thick,dashed] (0,-0.4) to [out=-70,in=70] (0,-2);
\draw[thick] (0,-0.4) to [out=-110,in=110] (0,-2);

\draw[thick] (-2.2,-0.2) to [out=-20,in=200] (2.5,-0.4);
\draw[thick, dashed] (-2.2,-0.2) to [out=-22,in=202] (2.5,-0.4);

\draw[thick] (-2.3,-0.3) to [out=-20,in=200] (5,-0.4);
\draw[thick, dashed] (-2.3,-0.3) to [out=-22,in=202] (5,-0.4);

\draw[thick] (-2.5,-0.4) to [out=-20,in=200] (8.5,-0.4);
\draw[thick, dashed] (-2.5,-0.4) to [out=-22,in=202] (8.5,-0.4);

\draw[thick, dashed] (-2.2,0.2) to [out=20,in=-200] (2.5,0.4);
\draw[thick] (-2.2,0.2) to [out=22,in=-202] (2.5,0.4);

\draw[thick, dashed] (-2.3,0.3) to [out=20,in=-200] (5,0.4);
\draw[thick] (-2.3,0.3) to [out=22,in=-202] (5,0.4);

\draw[thick, dashed] (-2.5,0.4) to [out=20,in=-200] (8.5,0.4);
\draw[thick] (-2.5,0.4) to [out=22,in=-202] (8.5,0.4);

\node at (-2.5,0) {\tiny{1}};
\node at (0,0) {\tiny{2}};
\node at (2.5,0) {\tiny{3}};
\node at (5,0) {\tiny{4}};
\node at (8.5,0) {\tiny{g}};

\node[above] at (1.25,-0.1) {$\gamma'_{13}$};
\node[above] at (3.75,-0.1) {$\gamma'_{14}$};
\node[above] at (7,-0.1) {$\gamma'_{1g}$};
\node[below] at (1.25,-0.1) {$\gamma_{13}$};
\node[below] at (3.75,-0.1) {$\gamma_{14}$};
\node[below] at (7,-0.1) {$\gamma_{1g}$};
\node [above] at (0,2) {$\beta$};
\node [below] at (0,-2) {$\beta'$};

\draw[thick,pattern=north west lines] (11,-2) to [out=130,in=-130] (11,2) to [out=-50,in=50] (11,-2);
\end{tikzpicture}
\end{center}
\caption{Curves involved in the definition of elements of $\mathcal{AB}_{g,1}$.}
\end{figure}
Fix an integer $j$ with $3\leq j\leq g.$ Consider $\beta''_j=(T_{\gamma_{1(j+1)}}^{x_{(j+1)}}\cdots T_{\gamma_{1g}}^{x_g})(\beta')$ for $3\leq j\leq g-1$ and $\beta''_g=\beta'.$ For $k\geq 1$ we have that
$$
T_\beta T_{T^k_{\gamma_{1j}}(\beta''_j)}^{-1}=  T_\beta T_{\beta'}^{-1}+T_{\beta'}T_{T^{-1}_{\gamma_{1j}'}(\beta)}^{-1}+ T_{T^{-1}_{\gamma_{1j}'}(\beta)}T_{T^k_{\gamma_{1j}}(\beta''_j)}^{-1}.$$
Since $T_{\gamma'_{1j}}T_{\gamma_{1j}}^{-1}\in \mathcal{AB}_{g,1}$ for $3\leq j\leq g,$ conjugating by $T_{\gamma'_{1j}}T_{\gamma_{1j}}^{-1}$ the last two terms of the above equation, in the coinvariant module, we get that 
\begin{equation}
\label{eq_red_twists}
T_\beta T_{T^k_{\gamma_{1j}}(\beta''_j)}^{-1}= T_\beta T_{\beta'}^{-1}+T_{T^{-1}_{\gamma_{1j}}(\beta')}T_{\beta}^{-1}+ T_{\beta}T_{T^{k-1}_{\gamma_{1j}}(\beta''_j)}^{-1}.
\end{equation}
Applying Eq. \eqref{eq_red_twists} from $k=x_j$ to $k=1,$ we obtain that
\begin{equation}
\label{eq_dif_CBP}
\begin{aligned}[b]
 T_\beta T_{T^{x_j}_{\gamma_{1j}}(\beta''_j)}^{-1}= & x_j T_\beta T_{\beta'}^{-1}+x_j T_{T^{-1}_{\gamma_{1j}}(\beta')}T_{\beta}^{-1}+ T_\beta T_{\beta''_j}^{-1} \\
= & x_j\big(T_\beta T_{\beta'}^{-1}-T_{\beta}T^{-1}_{T^{-1}_{\gamma_{1j}}(\beta')}\big)+ T_\beta T_{\beta''_j}^{-1}.
\end{aligned}
\end{equation}
Applying recursively Eq. \eqref{eq_dif_CBP} from $j=3$ to $j=g$, we get the following formula:
\begin{equation}
\label{eq_dif_CBP_general}
T_{\beta} T_{(T_{\gamma_{13}}^{x_3}\cdots T_{\gamma_{1g}}^{x_g})(\beta')}^{-1}= T_{\beta} T_{\beta'}^{-1}+\sum^g_{i=3}x_i\big(T_{\beta} T_{\beta'}^{-1}-T_{\beta} T_{T^{-1}_{\gamma_{1i}}(\beta')}^{-1}\big).
\end{equation}

In sequel we prove that for $3\leq k\leq g,$ in the coinvariant module,
$$T_\beta T_{\beta'}^{-1}=T_{\beta}T_{T^{-1}_{\gamma_{1k}}(\beta')}^{-1}.$$
Consider the element $f_k\in \mathcal{AB}_{g,1}$ given by the half twist of the shaded ball depicted in the following figure, that exchanges the holes $3$ and $k.$

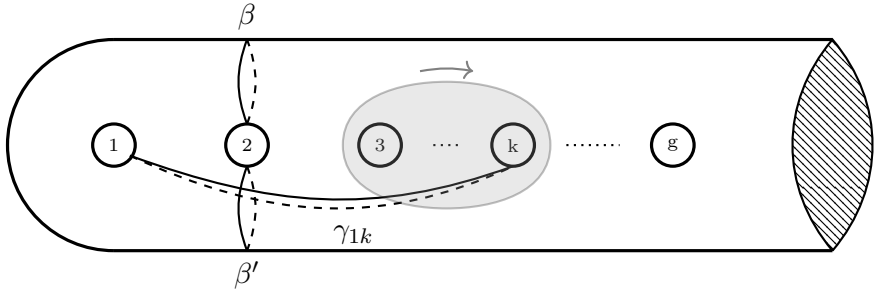
\begin{figure}[H]
\begin{center}
\begin{tikzpicture}[scale=.7]
\draw[very thick] (-2.5,-2) -- (11,-2);
\draw[very thick] (-2.5,2) -- (11,2);
\draw[very thick] (-2.5,2) arc [radius=2, start angle=90, end angle=270];

\draw[very thick] (0,0) circle [radius=.4];
\draw[very thick] (-2.5,0) circle [radius=.4];
\draw[very thick] (2.5,0) circle [radius=.4];
\draw[very thick] (5,0) circle [radius=.4];
\draw[very thick] (8,0) circle [radius=.4];

\draw[thick,dashed] (0,0.4) to [out=70,in=-70] (0,2);
\draw[thick] (0,0.4) to [out=110,in=-110] (0,2);

\draw[thick,dashed] (0,-0.4) to [out=-70,in=70] (0,-2);
\draw[thick] (0,-0.4) to [out=-110,in=110] (0,-2);

\draw[thick] (-2.2,-0.2) to [out=-20,in=200] (5,-0.4);
\draw[thick, dashed] (-2.2,-0.2) to [out=-25,in=205] (5,-0.4);

\draw[thick, dotted] (6,0) -- (7,0);
\draw[thick, dotted] (3.5,0) -- (4,0);

\node at (-2.5,0) {\tiny{1}};
\node at (0,0) {\tiny{2}};
\node at (2.5,0) {\tiny{3}};
\node at (5,0) {\tiny{k}};
\node at (8,0) {\tiny{g}};

\node [above] at (0,2) {$\beta$};
\node [below] at (0,-2) {$\beta'$};

\draw[thick, fill=gray!70, nearly transparent] (1.8,0) to [out=90,in=180] (3.75,1.2) to [out=0,in=90] (5.7,0) to [out=-90,in=0] (3.75,-1.2) to [out=180,in=-90] (1.8,0);

\draw[thick, gray, ->] (3.25,1.4) to [out=10,in=170] (4.25,1.4);

\node at (2,-1.7) {$\gamma_{1k}$};

\draw[thick,pattern=north west lines] (11,-2) to [out=130,in=-130] (11,2) to [out=-50,in=50] (11,-2);

\end{tikzpicture}
\end{center}
\caption{Half twist $f_k$ exchanging holes $3$ and $k$.}
\end{figure}

Since $f_k$ leaves $\beta,$ $\beta'$ invariant and sends $\gamma_{1k}$ to
$\gamma_{13},$ 
for $3\leq k\leq g,$ in the coinvariant module, we have that
\begin{equation}
\label{eq_CBP_igual_3}
\begin{aligned}[b]
T_{\beta}T_{T^{-1}_{\gamma_{1k}}(\beta')}^{-1}= &
T^{-1}_{\gamma_{1k}}T_{\beta}T_{\beta'}^{-1}T_{\gamma_{1k}}=
f_kT^{-1}_{\gamma_{1k}}f_k^{-1}T_{\beta}T_{\beta'}^{-1}f_kT_{\gamma_{1k}}f_k^{-1}= \\
= & T^{-1}_{\gamma_{13}}T_{\beta}T_{\beta'}^{-1}T_{\gamma_{13}}= T_{\beta}T_{T^{-1}_{\gamma_{13}}(\beta')}^{-1}.
\end{aligned}
\end{equation}
Therefore it is enough to show that in the coinvariant module,
$$T_\beta T_{\beta'}^{-1}=T_{\beta}T_{T^{-1}_{\gamma_{13}}(\beta')}^{-1}.$$
Since $\beta_1,$ $\beta_3$ are disjoint with $\beta,$ $\beta',$ $\gamma_{13},$ we have that 
\begin{equation*}
T_{\beta}T^{-1}_{T_{\gamma_{13}}(\beta')}=T_{\gamma_{13}}T_{\beta}T_{\beta'}^{-1}T_{\gamma_{13}}^{-1}=(T_{\beta_1}^{-1}T_{\gamma_{13}}T_{\beta_3}^{-1})T_{\beta}T_{\beta'}^{-1}(T_{\beta_1}^{-1}T_{\gamma_{13}}T_{\beta_3}^{-1})^{-1}.
\end{equation*}
Now take $f\in \mathcal{AB}_{g,1}$ given by $f=T_{\alpha_3}T_{\eta_{34}}^{-1}T_{\beta_{4}},$
where $\alpha_3,$ $\eta_{34},$ $\beta_{4}$ are the curves on $\Sigma_{g,1}$ given in the following picture:
\begin{figure}[H]
\begin{center}
\begin{tikzpicture}[scale=.7]
\draw[very thick] (-2.5,-2) -- (11,-2);
\draw[very thick] (-2.5,2) -- (11,2);
\draw[very thick] (-2.5,2) arc [radius=2, start angle=90, end angle=270];

\draw[very thick] (0,0) circle [radius=.4];
\draw[very thick] (-2.5,0) circle [radius=.4];
\draw[very thick] (2.5,0) circle [radius=.4];
\draw[very thick] (5,0) circle [radius=.4];
\draw[very thick] (8,0) circle [radius=.4];

\draw[thick, dotted] (6,0) -- (7,0);

\draw[thick] (2.5,0) circle [radius=.8];
\draw[thick] (2.5,1.2) arc [radius=1.2, start angle=90, end angle=360];
\draw[thick] (3.7,0) to [out=90,in=180] (4.3,0.8);

\draw[thick] (2.5,1.2) -- (4.3,1.2);

\draw[thick] (4.3,0.8) to [out=0,in=150] (4.8,0.3);
\draw[thick] (4.3,1.2) to [out=0,in=210] (4.8,2);
\draw[thick,dashed] (4.8,0.3) to [out=70,in=-70] (4.8,2);

\draw[thick,dashed] (5.2,0.3) to [out=70,in=-70] (5.2,2);
\draw[thick] (5.2,0.3) to [out=110,in=-110] (5.2,2);

\draw[thick,dashed] (0,0.4) to [out=70,in=-70] (0,2);
\draw[thick] (0,0.4) to [out=110,in=-110] (0,2);

\draw[thick,dashed] (0,-0.4) to [out=-70,in=70] (0,-2);
\draw[thick] (0,-0.4) to [out=-110,in=110] (0,-2);

\node [above] at (0,2) {$\beta$};
\node [below] at (0,-2) {$\beta'$};

\node at (3,1.5) {$\eta_{34}$};
\node at (3.4,0.7) {$\alpha_3$};
\node at (5.8,1) {$\beta_{4}$};

\node at (-2.5,0) {\tiny{1}};
\node at (0,0) {\tiny{2}};
\node at (2.5,0) {\tiny{3}};
\node at (5,0) {\tiny{4}};
\node at (8,0) {\tiny{g}};

\draw[thick,pattern=north west lines] (11,-2) to [out=130,in=-130] (11,2) to [out=-50,in=50] (11,-2);

\end{tikzpicture}
\end{center}
\caption{Curves involved in the definition of $f\in \mathcal{AB}_{g,1}$.}
\end{figure}
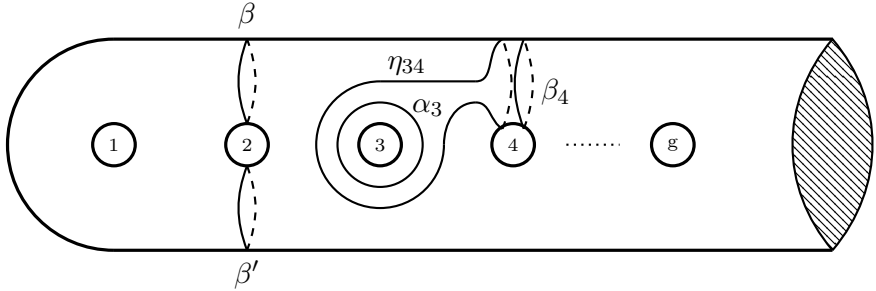
Since $\alpha_3,$ $\eta_{34},$ $\beta_4$ do not intersect either of $\beta,$ $\beta',$ the element $f$ commutes with $T_\beta T^{-1}_{\beta'}$ and in the coinvariant module we have that
\begin{equation}
\label{com-beta13f}
T_{\beta}T^{-1}_{T_{\gamma_{13}}(\beta')}=(f(T_{\beta_1}^{-1}T_{\gamma_{13}}T_{\beta_3}^{-1})f^{-1})T_{\beta}T_{\beta'}^{-1}(f(T_{\beta_1}^{-1}T_{\gamma_{13}}T_{\beta_3}^{-1})f^{-1})^{-1}.
\end{equation}
Observe that
\begin{align*}
 \Psi(f(T_{\beta_1}^{-1}T_{\gamma_{13}}T_{\beta_3}^{-1})f^{-1})= & \Psi(f)\Psi(T_{\beta_1}^{-1}T_{\gamma_{13}}T_{\beta_3}^{-1})\Psi(f^{-1}),\\
\Psi((T_{\beta_1}^{-1}T_{\gamma_{13}}T_{\beta_3}^{-1})(T_{\beta_1}^{-1}T_{\gamma_{14}}T_{\beta_4}^{-1}))= &\Psi(T_{\beta_1}^{-1}T_{\gamma_{13}}T_{\beta_3}^{-1})\Psi(T_{\beta_1}^{-1}T_{\gamma_{14}}T_{\beta_4}^{-1}),
\end{align*}
where
$$\Psi(T_{\beta_1}^{-1}T_{\gamma_{13}}T_{\beta_3}^{-1})=\left(\begin{matrix}
Id & 0 \\
SE_{13} & Id
\end{matrix}\right), \qquad \Psi(T_{\beta_1}^{-1}T_{\gamma_{14}}T_{\beta_4}^{-1})=\left(\begin{matrix}
Id & 0 \\
SE_{14} & Id
\end{matrix}\right),$$
$$\Psi(f)=\Psi(T_{\alpha_3}T_{\eta_{34}}^{-1}T_{\beta_{4}})=\left(\begin{matrix}
Id-E_{34} & 0 \\
0 & Id+E_{43}
\end{matrix}\right),$$
$$\Psi(f^{-1})=\Psi(f)^{-1}=\left(\begin{matrix}
Id+E_{34} & 0 \\
0 & Id-E_{43}
\end{matrix}\right).$$
A direct computation shows that
$$\Psi(f(T_{\beta_1}^{-1}T_{\gamma_{13}}T_{\beta_3}^{-1})f^{-1})= \Psi((T_{\beta_1}^{-1}T_{\gamma_{13}}T_{\beta_3}^{-1})(T_{\beta_1}^{-1}T_{\gamma_{14}}T_{\beta_4}^{-1})).$$
Then, by the short exact sequence \eqref{ses_L}, there is an element $\xi_b\in \mathcal{LTB}_{g,1}$ such that
\begin{equation}
\label{eq_gamma_13,14}
f(T_{\beta_1}^{-1}T_{\gamma_{13}}T_{\beta_3}^{-1})f^{-1}=\xi_b (T_{\beta_1}^{-1}T_{\gamma_{13}}T_{\beta_3}^{-1})(T_{\beta_1}^{-1}T_{\gamma_{14}}T_{\beta_4}^{-1}).
\end{equation}
Since $T_{\beta_1},$ $T_{\beta_3},$ $T_{\beta_4}$ commute with $T_{\gamma_{13}},$ $T_{\gamma_{14}},$ $T_\beta,$ $T_{\beta'},$ and $f$ commutes with $T_\beta,$ $T_{\beta'},$ by the Eqs. \eqref{com-beta13f} and \eqref{eq_gamma_13,14}, in the coinvariant module, we get that 
\begin{equation}
\label{eq_CBP_13_to_1314}
T_{\beta}T^{-1}_{T_{\gamma_{13}}(\beta')}
= (T_{\gamma_{13}}T_{\gamma_{14}})T_{\beta}T_{\beta'}^{-1}(T_{\gamma_{13}}T_{\gamma_{14}})^{-1}=  T_{\beta}T_{T_{\gamma_{13}}T_{\gamma_{14}}(\beta')}^{-1}.
\end{equation}
Notice that we have the following equation:
$$T_{\beta}T_{T_{\gamma_{13}}T_{\gamma_{14}}(\beta')}^{-1}= T_{\beta}T_{\beta'}^{-1}+T_{\beta'} T^{-1}_{T_{\gamma'_{13}}^{-1}(\beta)}+T_{T_{\gamma'_{13}}^{-1}(\beta)}T_{T_{\gamma_{13}}T_{\gamma_{14}}(\beta')}^{-1}.$$
Conjugating the last two terms by $T_{\gamma_{13}}T_{\gamma'_{13}}^{-1}\in \mathcal{AB}_{g,1},$ in the coinvariant module,
$$T_{\beta}T_{T_{\gamma_{13}}T_{\gamma_{14}}(\beta')}^{-1}= T_{\beta}T_{\beta'}^{-1}+T_{T_{\gamma_{13}}^{-1}(\beta')} T^{-1}_{\beta} + T_{\beta} T^{-1}_{T_{\gamma_{14}}(\beta')},$$
and conjugating the last term by $f_4\in\mathcal{AB}_{g,1},$ in the coinvariant module,
\begin{equation}
\label{eq_CBP_1314_beta}
T_{\beta}T_{T_{\gamma_{13}}T_{\gamma_{14}}(\beta')}^{-1}=T_{\beta}T_{\beta'}^{-1}- T_{\beta}T_{T_{\gamma_{13}}^{-1}(\beta')}^{-1} + T_{\beta} T^{-1}_{T_{\gamma_{13}}(\beta')}.
\end{equation}
Hence, by the Eqs. \eqref{eq_CBP_13_to_1314} and
\eqref{eq_CBP_1314_beta},
\begin{equation*}
T_{\beta}T^{-1}_{T_{\gamma_{13}}(\beta')}=T_{\beta}T_{\beta'}^{-1}- T_{\beta}T_{T_{\gamma_{13}}^{-1}(\beta')}^{-1} + T_{\beta} T^{-1}_{T_{\gamma_{13}}(\beta')}.
\end{equation*}
Then in the coinvariant module $T_{\beta}T_{\beta'}^{-1}= T_{\beta}T_{T_{\gamma_{13}}^{-1}(\beta')}^{-1}$ and by the Eqs. \eqref{eq_CBP_igual_3}, \eqref{eq_dif_CBP_general} and \eqref{eq_CBP_product},
$T_\nu T_{\nu'}^{-1}$ is equivalent to $T_{\beta} T_{\beta'}^{-1},$
and hence all CBP-twists of genus one are equivalent.

\;

\textbf{Step 2.} We show that all CBP-twists of genus $1$ are equivalent to zero using Step 1 and the lantern relation.

First of all notice that all CBP-twists of genus one have $2$-torsion since by Step 1, $T_{\nu}T_{\nu'}^{-1}=T_{\nu'}T_{\nu}^{-1}.$ Then, any CBP-twists of genus two is equivalent to zero since by Proposition \ref{prop_CBP_1}, any CBP-twists of genus two is a product of two CBP-twists of genus one.

Consider the following curves in the standardly embedded surface $\Sigma_{g,1}:$

\begin{figure}[H]
\begin{center}
\begin{tikzpicture}[scale=.7]
\draw[very thick] (-4.5,-2) -- (10,-2);
\draw[very thick] (-4.5,2) -- (10,2);
\draw[very thick] (-4.5,2) arc [radius=2, start angle=90, end angle=270];
\draw[very thick] (-4.5,0) circle [radius=.4];
\draw[very thick] (-2,0) circle [radius=.4];

\draw[very thick] (4.5,0) circle [radius=.4];
\draw[very thick] (7,0) circle [radius=.4];

\draw[thick, dashed] (-2,0.4) to [out=0,in=0] (-2,2);
\draw[thick, dashed] (-2,-0.4) to [out=0,in=0] (-2,-2);
\draw[thick, dashed] (-4.5,-0.4) to [out=0,in=0] (-4.5,-2);

\draw[thick] (-4.5,-0.4) to [out=180,in=90] (-5,-1.2);
\draw[<-,thick](-5,-1.2) to [out=-90,in=180] (-4.5,-2);
\draw[thick] (-2,-0.4) to [out=180,in=90] (-2.5,-1.2);
\draw[<-,thick](-2.5,-1.2) to [out=-90,in=180] (-2,-2);
\draw[thick] (-2,0.4) to [out=180,in=-90] (-2.5,1.2);
\draw[<-,thick](-2.5,1.2) to [out=90,in=180] (-2,2);

\draw[thick, dashed] (-4.5,0.4) to [out=30,in=180] (-2,1);
\draw[thick, dashed] (-2,1) to [out=0,in=90] (-0.7,0);
\draw[thick, dashed] (-0.7,0) to [out=-90,in=50] (-1,-2);

\draw[thick] (-4.5,0.4) to [out=0,in=180] (-2,0.7);
\draw[thick] (-2,0.7) to [out=0,in=90] (-1.3,0);
\draw[thick] (-1.3,0) to [out=-90,in=130] (-1,-2);

\draw[thick] (-4.15,-0.2) to [out=-20,in=200] (-2.35,-0.2);
\draw[thick, dashed] (-4.15,-0.2) to [out=20,in=160] (-2.35,-0.2);

\draw[thick, dashed] (-4.15,0.2) to [out=20,in=160] (-2.35,0.2);

\draw[thick] (-4.15,0.2) to [out=-20,in=-90] (-3.5,0.5);
\draw[thick] (-3.5,0.5) to [out=90,in=180] (-2,1.5);
\draw[thick] (-2,1.5) to [out=0,in=180] (-0.5,1.5);
\draw[thick] (-0.5,1.5) to [out=0,in=220] (0,2);
\draw[dashed,thick] (0,2) to [out=-50,in=90] (0.5,0) to [out=-90,in=50] (0,-2);
\draw[thick] (-2.35,0.2) to [out=200,in=-90] (-3,0.5);
\draw[thick] (-3,0.5) to [out=90,in=180] (-2,1.1);
\draw[thick] (-2,1.1) to [out=0,in=180] (-1,1.1);
\draw[thick] (-1,1.1) to [out=0,in=140] (0,-2);

\node [above] at (-2,2) {$\zeta'_1$};
\node [below] at (-2,-2) {$\zeta_1$};
\node [below] at (-1,-2) {$\zeta'_2$};
\node [below] at (-4.5,-2) {$\zeta_2$};
\node [above] at (0,2) {$\zeta'_3$};
\node [below] at (-3.25,-0.3) {$\zeta_3$};

\draw[dashed,thick] (1.5,2) to [out=-50,in=90] (2,0) to [out=-90,in=50] (1.5,-2);
\draw[thick] (1.5,-2) to [out=130,in=-90] (1,0);
\draw[->,thick] (1.5,2) to [out=230,in=90] (1,0);

\node [above] at (1.5,2) {$\gamma$};

\draw[thick] (7,0.4) to [out=180,in=-90] (6.5,1.2);
\draw[<-,thick](6.5,1.2) to [out=90,in=180] (7,2);
\draw[thick] (4.5,-0.4) to [out=180,in=90] (4,-1.2);
\draw[<-,thick](4,-1.2) to [out=-90,in=180] (4.5,-2);
\draw[thick] (4.5,0.4) to [out=180,in=-90] (4,1.2);
\draw[<-,thick](4,1.2) to [out=90,in=180] (4.5,2);

\draw[thick, dashed] (4.5,0.4) to [out=0,in=0] (4.5,2);
\draw[thick, dashed] (4.5,-0.4) to [out=0,in=0] (4.5,-2);
\draw[thick, dashed] (7,0.4) to [out=0,in=0] (7,2);

\draw[thick] (4.85,-0.2) to [out=-20,in=200] (6.65,-0.2);
\draw[thick, dashed] (4.85,-0.2) to [out=20,in=160] (6.65,-0.2);

\draw[thick, dashed] (4.85,0.2) to [out=20,in=160] (6.65,0.2);
\draw[thick] (4.85,0.2) to [out=-20,in=-90] (5.5,0.5);
\draw[thick] (5.5,0.5) to [out=90,in=0] (4.5,1.1);
\draw[thick] (4.5,1.1) to [out=180,in=0] (3,1.1);
\draw[thick] (3,1.1) to [out=180,in=130] (2.8,-2);

\draw[thick] (6.65,0.2) to [out=200,in=-90] (6,0.5);
\draw[thick] (6,0.5) to [out=90,in=0] (4.5,1.5);
\draw[thick] (4.5,1.5) to [out=180,in=0] (2.5,1.5);
\draw[thick] (2.5,1.5) to [out=180,in=230] (2.8,2);
\draw[dashed,thick] (2.8,2) to [out=-50,in=90] (3.3,0) to [out=-90,in=50] (2.8,-2);

\draw[thick] (7,0.4) to [out=150,in=0] (4.5,1);
\draw[thick] (4.5,1) to [out=180,in=90] (3,0);
\draw[thick] (3,0) to [out=-90,in=130] (3.5,-2);

\draw[thick, dashed] (7,0.4) to [out=180,in=0] (4.5,0.7);
\draw[thick, dashed] (4.5,0.7) to [out=180,in=90] (3.8,0);
\draw[thick, dashed] (3.8,0) to [out=-90,in=50] (3.5,-2);

\node [below] at (4.5,-2) {$\xi_1$};
\node [above] at (7,2) {$\xi_2$};
\node [above] at (3,2) {$\xi'_3$};
\node [above] at (4.5,2) {$\xi'_1$};
\node [below] at (3.5,-2) {$\xi'_2$};
\node [below] at (5.75,-0.3) {$\xi_3$};

\draw[thick,pattern=north west lines] (10,-2) to [out=130,in=-130] (10,2) to [out=-50,in=50] (10,-2);

\end{tikzpicture}
\end{center}
\caption{Two lantern configurations embedded in $\Sigma_{g,1}$.}
\end{figure}
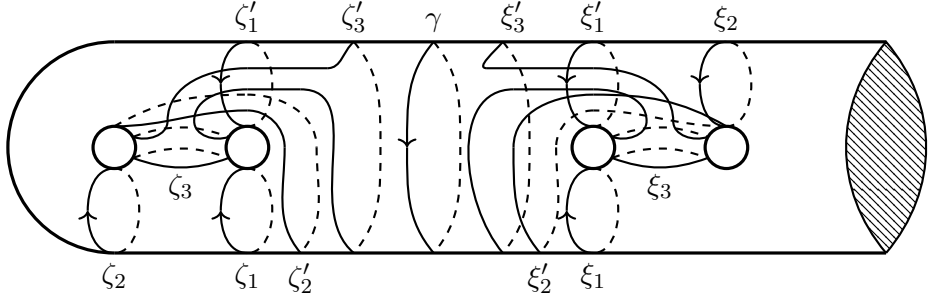

Observe that for $i=1,2,3,$ $T_{\zeta_i}T^{-1}_{\zeta'_i}$ are CBP-twists of genus $1,$ $T_{\xi_i}T^{-1}_{\xi'_i}$ are CBP-twists of genus $2.$
Consider the following lantern relations:
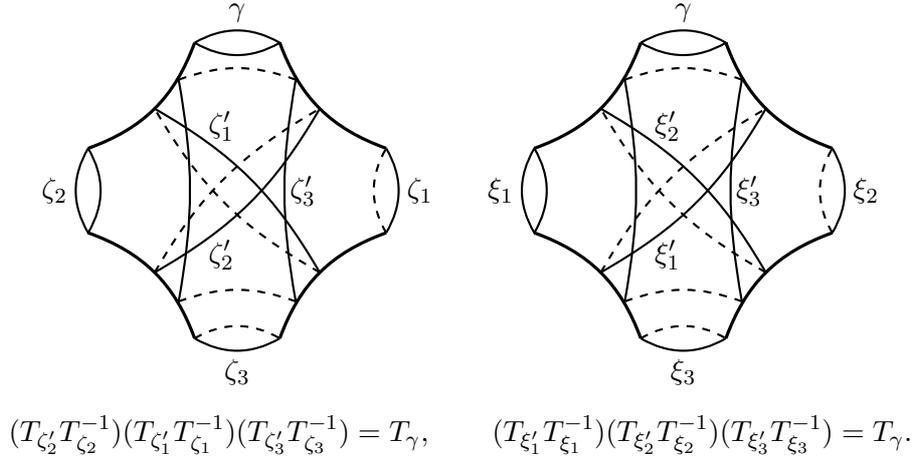
\begin{figure}[H]
\begin{center}
\begin{tikzpicture}[scale=.7]
\draw[very thick] (-2.8,0.8) to [out=20,in=-110] (-0.8,2.8);
\draw[very thick] (2.8,0.8) to [out=160,in=-70] (0.8,2.8);
\draw[very thick] (-2.8,-0.8) to [out=-20,in=110] (-0.8,-2.8);
\draw[very thick] (2.8,-0.8) to [out=-160,in=70] (0.8,-2.8);

\draw[thick] (-2.8,0.8) to [out=-60,in=60] (-2.8,-0.8) to [out=120,in=-120] (-2.8,0.8);
\draw[thick] (-0.8,2.8) to [out=30,in=150] (0.8,2.8) to [out=-150,in=-30] (-0.8,2.8);
\draw[thick] (2.8,0.8) to [out=-60,in=60] (2.8,-0.8);
\draw[thick, dashed] (2.8,-0.8) to [out=120,in=-120] (2.8,0.8);
\draw[thick, dashed] (-0.8,-2.8) to [out=30,in=150] (0.8,-2.8);
\draw[thick] (0.8,-2.8) to [out=-150,in=-30] (-0.8,-2.8);

\draw[thick, dashed] (-1.55,-1.55) to [out=60,in=210] (1.55,1.55);
\draw[thick] (-1.55,-1.55) to [out=30,in=240] (1.55,1.55);

\draw[thick] (1.55,-1.55) to [out=120,in=-30] (-1.55,1.55);
\draw[thick, dashed] (1.55,-1.55) to [out=150,in=-60] (-1.55,1.55);

\draw[thick] (-1.1,2.1) to [out=-80,in=80] (-1.1,-2.1);
\draw[thick] (1.1,2.1) to [out=-100,in=100] (1.1,-2.1);

\draw[thick, dashed] (-1.1,2.1) to [out=20,in=160] (1.1,2.1);
\draw[thick, dashed] (-1.1,-2.1) to [out=20,in=160] (1.1,-2.1);

\node [left] at (-3,0) {$\zeta_2$};
\node [right] at (3,0) {$\zeta_1$};
\node [above] at (0,3) {$\gamma$};
\node [below] at (0,-3) {$\zeta_3$};

\node [right] at (0.8,0) {$\zeta'_3$};
\node [below] at (-0.3,1.7) {$\zeta'_1$};
\node [above] at (-0.3,-1.7) {$\zeta'_2$};

\end{tikzpicture}
\;\;
\begin{tikzpicture}[scale=.7]
\draw[very thick] (-2.8,0.8) to [out=20,in=-110] (-0.8,2.8);
\draw[very thick] (2.8,0.8) to [out=160,in=-70] (0.8,2.8);
\draw[very thick] (-2.8,-0.8) to [out=-20,in=110] (-0.8,-2.8);
\draw[very thick] (2.8,-0.8) to [out=-160,in=70] (0.8,-2.8);

\draw[thick] (-2.8,0.8) to [out=-60,in=60] (-2.8,-0.8) to [out=120,in=-120] (-2.8,0.8);
\draw[thick] (-0.8,2.8) to [out=30,in=150] (0.8,2.8) to [out=-150,in=-30] (-0.8,2.8);
\draw[thick] (2.8,0.8) to [out=-60,in=60] (2.8,-0.8);
\draw[thick, dashed] (2.8,-0.8) to [out=120,in=-120] (2.8,0.8);
\draw[thick, dashed] (-0.8,-2.8) to [out=30,in=150] (0.8,-2.8);
\draw[thick] (0.8,-2.8) to [out=-150,in=-30] (-0.8,-2.8);

\draw[thick, dashed] (-1.55,-1.55) to [out=60,in=210] (1.55,1.55);
\draw[thick] (-1.55,-1.55) to [out=30,in=240] (1.55,1.55);

\draw[thick] (1.55,-1.55) to [out=120,in=-30] (-1.55,1.55);
\draw[thick, dashed] (1.55,-1.55) to [out=150,in=-60] (-1.55,1.55);

\draw[thick] (-1.1,2.1) to [out=-80,in=80] (-1.1,-2.1);
\draw[thick] (1.1,2.1) to [out=-100,in=100] (1.1,-2.1);

\draw[thick, dashed] (-1.1,2.1) to [out=20,in=160] (1.1,2.1);
\draw[thick, dashed] (-1.1,-2.1) to [out=20,in=160] (1.1,-2.1);

\node [left] at (-3,0) {$\xi_1$};
\node [right] at (3,0) {$\xi_2$};
\node [above] at (0,3) {$\gamma$};
\node [below] at (0,-3) {$\xi_3$};

\node [right] at (0.8,0) {$\xi'_3$};
\node [below] at (-0.3,1.7) {$\xi'_2$};
\node [above] at (-0.3,-1.7) {$\xi'_1$};

\end{tikzpicture}
$$(T_{\zeta'_2}T^{-1}_{\zeta_2})(T_{\zeta'_1}T^{-1}_{\zeta_1})(T_{\zeta'_3}T^{-1}_{\zeta_3})=T_\gamma,
\qquad
(T_{\xi'_1}T^{-1}_{\xi_1})(T_{\xi'_2}T^{-1}_{\xi_2})(T_{\xi'_3}T^{-1}_{\xi_3})=T_\gamma.$$
\end{center}
\caption{Lantern configurations with the induced lantern relations.}
\end{figure}
Putting these relations together we get that
$$(T_{\zeta'_1}T^{-1}_{\zeta_1})=(T_{\zeta_2}T^{-1}_{\zeta'_2})(T_{\zeta_3}T^{-1}_{\zeta'_3})(T_{\xi'_1}T^{-1}_{\xi_1})(T_{\xi'_2}T^{-1}_{\xi_2})(T_{\xi'_3}T^{-1}_{\xi_3}).$$
Since all CBP-twists of genus one are equivalent and all CBP-twists of genus two are equivalent to zero, $T_{\zeta'_1}T^{-1}_{\zeta_1}$ is equivalent to zero, and we conclude by Step 1.
\end{proof}

\subsubsection*{Proof of Theorem \ref{teo-roh-intro}}
Let $A$ be an abelian group, recall that we denote by $A_2$ the subgroup of $2$-torsion elements of $A$. For a given integer $g\geq 4,$ consider the BCJ-homomorphism $\sigma: \mathcal{T}_{g,1}\rightarrow \mathfrak{B}_3,$ the projection $\pi_g:\mathfrak{B}_3\rightarrow\mathfrak{B}_0\cong \mathbb{Z}/2$ and the injection $\varepsilon^x:\mathfrak{B}_0\rightarrow A_2$ defined by sending $\overline{1}$ to $x\in A_2.$ By Proposition \ref{prop-Torelli-GL}, composing the pull-back of the aforementioned homomorphisms we get the following sequence of isomorphisms:
$$A_2\xrightarrow[\sim]{(\epsilon^x)^*}Hom (\mathbb{Z}/2,A)\xrightarrow[\sim]{\sigma^*\circ \; \pi^*} Hom(H_1(\mathcal{T}_{g,1};\mathbb{Z})_{\mathcal{AB}_{g,1}},A)=Hom(\mathcal{T}_{g,1},A)^{\mathcal{AB}_{g,1}}.$$
Therefore we get an isomorphism
\begin{align*}
\Lambda: A_2 & \xrightarrow{\;\sim\;} Hom(\mathcal{T}_{g,1},A)^{\mathcal{AB}_{g,1}} \\
x & \longmapsto \mu^x_g= \varepsilon^{x}\circ \pi_g \circ\sigma  .
\end{align*}

We show that the family of homomorphisms $\{\mu_g^x\}_g$ reassemble into the Rohklin invariant.
By the constructions of $\sigma,$ $\pi_g$ and $\varepsilon^x,$ these maps are compatible with the stabilization map and then the compositions of these maps $\{\mu_g^x\}_g$ are also compatible with the stabilization map.
By Lemma \ref{lem_TB,TA}, the $\mathcal{AB}_{g,1}$-invariant homomorphisms $\{\mu_g^x\}_g$ are zero on $\mathcal{TA}_{g,1},$ $\mathcal{TB}_{g,1}.$
Then, by the bijection \eqref{bij_S3}, the family of homomorphism $\{\mu^x_g\}_g$ reassemble into an invariant of integral homology $3$-spheres.
Besides, precomposing the Rohlin invariant $R:\mathcal{S}^3\rightarrow \mathbb{Z}/2$ with the bijection \eqref{bij_S3} we get a family of homomorphisms $\{R_g\}_g$ with $R_g\in Hom(\mathcal{T}_{g,1};\mathbb{Z}/2)^{\mathcal{AB}_{g,1}}.$
Since there is only one non-zero element in $Hom(\mathcal{T}_{g,1};\mathbb{Z}/2)^{\mathcal{AB}_{g,1}},$ by Proposition \ref{prop-Torelli-GL}, $\pi_g\circ \sigma$ and $R_g$ must coincide. Therefore, $\mu_g^x$ and $\varepsilon^x \circ R_g$ must coincide too.

\section{Application}

As we learnt from \cite{pitsch}, for a given invariant of integral homology $3$-spheres $F:\mathcal{S}^3\rightarrow A$ there is an associated family of trivial $2$-cocycles $\{C_g\}_g$ on the Torelli group which measure the failure of the maps $\{F_g\}_g$ to be homomorphisms of groups,
\begin{align*}
C_g: \mathcal{T}_{g,1}\times \mathcal{T}_{g,1} & \longrightarrow A \\
 (\phi,\psi) & \longmapsto F_g(\phi)+F_g(\psi)-F_g(\phi\psi).
\end{align*}
Since $F$ is an invariant, this family of $2$-cocycles inherits the following properties:
\begin{enumerate}[(1)]
\item The $2$-cocycles $\{C_g\}_g$ are compatible with the stabilization map,
\item The $2$-cocycles $\{C_g\}_g$ are invariant under conjugation by elements in $\mathcal{AB}_{g,1},$
\item If $\phi\in \mathcal{TA}_{g,1}$ or $\psi \in \mathcal{TB}_{g,1}$ then $C_g(\phi, \psi)=0.$
\end{enumerate}
In the sequel, we show that for a given family of trivial 2-cocicles satisfying the conditions (1)-(3) and with a $\mathcal{AB}_{g,1}$-invariant trivialization, this family induces $A_2$-valued invariants of integral homology $3$-spheres.

First we show two other ways of expressing the condition about the existence of a $\mathcal{AB}_{g,1}$-invariant trivialization.
Consider $\mathcal{Q}_{C_g}$ the set of all trivializations of the 2-cocycle $C_g:$
$$\mathcal{Q}_{C_g}=\{q:\mathcal{T}_{g,1}\rightarrow A\mid q(\phi)+q(\psi)-q(\phi\psi)=C_g(\phi,\psi)\}.$$
Recall that any two trivializations of a given $2$-cocycle differ by an element of $Hom(\mathcal{T}_{g,1},A).$
As the cocycle $C_g$ is invariant under conjugations by $\mathcal{AB}_{g,1},$ this group acts on $\mathcal{Q}_{C_g}$ via its conjugation action on the Torelli group and it confers the set $\mathcal{Q}_{C_g}$ the structure of an affine set over the abelian group $Hom(\mathcal{T}_{g,1},A).$ Then the existence of an $\mathcal{AB}_{g,1}$-invariant trivialization is equivalent to the existence of a fixed point for the action of $\mathcal{AB}_{g,1} \in \mathcal{Q}_{C_g}.$ This condition can be also seen as a cohomological condition:
choose an arbitrary element $q\in \mathcal{Q}_{C_g}$ and define a map as follows
\begin{align*}
\rho_q:\mathcal{AB}_{g,1} & \longrightarrow Hom(\mathcal{T}_{g,1},A)\\ \phi & \longmapsto \phi \cdot q-q.
\end{align*}
A direct computation shows that $\rho_q$ is a derivation and the difference $\rho_q-\rho_{q'}$ for two elements in $\mathcal{Q}_{C_g}$ is a principal derivation. Therefore there is a well-defined cohomology class
$$\rho(C_g)\in H^1(\mathcal{AB}_{g,1};Hom(\mathcal{T}_{g,1},A))$$
called the torsor of the cocycle $C_g.$

\begin{proposition}
The natural action of $\mathcal{AB}_{g,1}$ on $\mathcal{Q}_{C_g}$ admits a fixed point if and only if the associated torsor $\rho(C_g)$ is trivial.
\end{proposition}

Finally, we finish this section with the proof of the main theorem of this paper.

\subsubsection*{Proof of Theorem \ref{teo_gen_tool-intro}}
Suppose that for every $g\geq 4$ there is a fixed point $q_g$ of $\mathcal{Q}_{C_g}$ for the action of $\mathcal{AB}_{g,1}$ on $\mathcal{Q}_{C_g}.$
Since every pair of $\mathcal{AB}_{g,1}$-invariant trivializations differ by an $\mathcal{AB}_{g,1}$-invariant homomorphism, by Theorem \ref{teo-roh-intro} the fixed points are $q_g+\mu_g^x$ with $x\in A_2.$

By construction, the family $\{\mu^x_g\}_g$ is compatible with the stabilization map. Then, given two different fixed points $q_g,$ $q'_g$ of $\mathcal{Q}_{C_g}$ for the action of $\mathcal{AB}_{g,1},$ we have that the following equation holds:
$${q_g|}_{\mathcal{T}_{g-1,1}}-{q'_g|}_{\mathcal{T}_{g-1,1}}=(q_g-q'_g)|_{\mathcal{T}_{g-1,1}} ={\mu_g^x|}_{\mathcal{T}_{g-1,1}}=\mu_{g-1}^x.$$
Therefore the restriction of the trivializations of $\mathcal{Q}_{C_g}$ to $\mathcal{T}_{g-1,1}$ give us a bijection between the fixed points of $\mathcal{Q}_{C_g}$ for the action of $\mathcal{AB}_{g,1}$ and the fixed points of $\mathcal{Q}_{C_{g-1}}$ for the action of $\mathcal{AB}_{g-1,1}.$ Consequently, for a given $\mathcal{AB}_{g,1}$-invariant trivialization $q_g$ and each $x\in A_2,$ we get a well-defined map
$$
q+\mu^x= \lim_{g\to \infty}q_g+\mu^x_g: \lim_{g\to \infty}\mathcal{T}_{g,1}\longrightarrow A.
$$
These are the only candidates to be $A$-valued invariants of integral homology $3$-spheres with associated family of $2$-cocycles $\{C_g\}_g.$ For these maps to be invariants, since they are already $\mathcal{AB}_{g,1}$-invariant, we only have to prove that they are well-defined on the double cosets $\mathcal{TA}_{g,1}\backslash \mathcal{T}_{g,1}/\mathcal{TB}_{g,1}.$
From property (3) of our cocycle we have that $\forall \phi\in \mathcal{T}_{g,1},$
$\forall \psi_a\in \mathcal{TA}_{g,1}$ and $\forall \psi_b\in \mathcal{TB}_{g,1},$
\begin{align}
(q_g+\mu^x_g)(\phi)-(q_g+\mu^x_g)(\phi \psi_b)= & -(q_g+\mu^x_g)(\psi_b) , \label{eq_TB_constant} \\
(q_g+\mu^x_g)(\phi)-(q_g+\mu^x_g)(\psi_a\phi )= & -(q_g+\mu^x_g)(\psi_a). \label{eq_TA_constant}
\end{align}
In particular, taking $\phi\in \mathcal{TB}_{g,1}$ in \eqref{eq_TB_constant} and $\phi\in \mathcal{TA}_{g,1}$ in \eqref{eq_TA_constant}, we get that $q_g+\mu^x_g$ are $\mathcal{AB}_{g,1}$-invariant homomorphisms on $\mathcal{TA}_{g,1}$ and $\mathcal{TB}_{g,1}$ and by Lemma \ref{lem_TB,TA}, the maps $q_g+\mu^x_g$ are zero on these last two groups.
Therefore the maps $q_g+\mu_g^x$ are well-defined on the double coset $\mathcal{TA}_{g,1}\backslash \mathcal{T}_{g,1}/\mathcal{TB}_{g,1}.$ 

\section*{Acknowledgments}
The author would like to express his gratitude to Prof. Wolfgang Pitsch from Universitat Autònoma de Barcelona for his encouragement and helpful advices throughout all this work.
This work was partially supported by MEC grant MTM2016-80439-P.

\pagebreak

\bibliography{biblio}{}
\bibliographystyle{abbrv}

\end{document}